\documentclass[preprint,12pt,numbers]{elsarticle}

\usepackage[T1]{fontenc}
\usepackage[utf8]{inputenc}
\usepackage{amsmath,amssymb,amsfonts,amsthm,mathtools}
\usepackage{enumitem}
\usepackage{microtype}
\usepackage[colorlinks=true,linkcolor=blue,citecolor=blue,urlcolor=blue]{hyperref}
\usepackage{tikz}
\usetikzlibrary{arrows.meta,positioning,fit,calc,backgrounds}
\usepackage{float}

\setlist[enumerate]{label=(\roman*),leftmargin=2.2em}
\setlist[itemize]{leftmargin=2.2em}

\newtheorem{theorem}{Theorem}[section]
\newtheorem*{theorem*}{Theorem}
\newtheorem{lemma}[theorem]{Lemma}
\newtheorem{proposition}[theorem]{Proposition}
\newtheorem{corollary}[theorem]{Corollary}

\theoremstyle{definition}
\newtheorem{definition}[theorem]{Definition}
\newtheorem{remark}[theorem]{Remark}
\newtheorem{example}[theorem]{Example}

\newcommand{\RL}{\mathsf{RL}}
\newcommand{\BDL}{\mathsf{BDL}}
\newcommand{\BA}{\mathsf{BA}}
\newcommand{\HA}{\mathsf{HA}}
\newcommand{\GA}{\mathsf{GA}}
\newcommand{\MV}{\mathsf{MV}}
\newcommand{\DHA}{\mathsf{DHA}}
\newcommand{\MTL}{\mathsf{MTL}}
\newcommand{\BL}{\mathsf{BL}}

\newcommand{\Fin}{\mathsf{Fin}}
\newcommand{\Ch}{\mathsf{Ch}}
\newcommand{\Root}{\mathsf{Root}}
\newcommand{\Spec}{\operatorname{Spec}}

\newcommand{\Ret}{\operatorname{Ret}}

\newcommand{\Con}{\operatorname{Con}}
\newcommand{\ClopUp}{\operatorname{ClopUp}}
\newcommand{\Up}{\operatorname{Up}}
\newcommand{\Id}{\operatorname{Id}}
\newcommand{\Idem}{\operatorname{Idem}}

\journal{Advances in Mathematics}

\begin{document}

\begin{frontmatter}

\title{Stone profiniteness, reticulation spectra, and profinite-completion obstructions for residuated lattices}

\author[addr1]{Jiang Yang}
\ead{yangjiangdy@126.com}

\address[addr1]{School of Mathematical Sciences, Guangxi Minzu University, Nanning, 530006, China}

\begin{abstract}
We develop a profinite theory for bounded commutative integral residuated lattices around three related questions.  First, we prove a Stone--profinite theorem: a topological residuated lattice is profinite exactly when its underlying space is Stone.  The proof is not merely reduct-level.  Starting from clopen congruences of the semiring reduct $(A,\vee,\odot,0,1)$, we construct cofinal clopen congruences in the full residuated-lattice signature via the $1$-class filter.

Second, we solve the classwise Gr\"atzer-type spectral representation problem in its natural residuated-lattice form.  For every class $\mathcal K$ of residuated lattices, the $\mathcal K$-representable prime-filter spectra are precisely the Priestley spectra of the bounded distributive lattices in the reticulation image $\Ret(\mathcal K)$.  This reduction is sharp: arbitrary classes of finite bounded distributive lattices occur as reticulation images of classes of finite residuated lattices.  We then compute the reticulation image in several algebraically meaningful cases.  In particular, finite MTL- and BL-representable posets are exactly finite root systems, and in $n$-potent, equivalently EDPC, settings the reticulation is the internal power skeleton $A^{[n]}=\{a^n:a\in A\}$.

Third, we study profinite completions.  We give an intrinsic criterion: a profinite residuated lattice is a profinite completion precisely when it has a dense subalgebra whose finite-index congruences are exactly those induced by clopen congruences of the ambient profinite algebra.  This yields a positive theorem for $n$-potent varieties when finite-index congruences of the dense subalgebra are finitely generated.  We also prove a finite-tag obstruction theorem: if a finite residuated lattice $T$ is invisible to finite Heyting targets, then $H\times T$ can be a profinite completion only if the profinite Heyting algebra $H$ is itself a Heyting profinite completion.  Taking $T=\mathbf L_3$ and using the Heyting non-completion examples of Bezhanishvili--Bezhanishvili--Moraschini--Stronkowski gives genuinely non-idempotent profinite residuated lattices which are not profinite completions.  Thus the diamond-Heyting criterion remains a Heyting phenomenon, while general residuated lattices require reticulation and finite-tag mechanisms.
\end{abstract}


\end{frontmatter}

\section{Introduction}

The aim of this paper is to isolate and unify the residuated-lattice manifestations of three closely related profinite phenomena.  The first asks when a topological residuated lattice is profinite purely for topological reasons.  The second asks which posets occur as prime-filter spectra of residuated lattices, in the spirit of Gr\"atzer's representation problem for prime spectra of distributive lattices.  The third asks which profinite residuated lattices arise as profinite completions of abstract residuated lattices.  These questions are connected, but they are not the same: Stone profiniteness concerns finite clopen quotients of a topological algebra, spectral representability concerns reticulation and Priestley duality, while profinite completion asks whether a given profinite object is generated, in the profinite sense, by an abstract algebra with exactly the prescribed finite quotients.
We deliberately separate the last question from a full classification problem: the paper does not claim to list all profinite residuated lattices that are completions.  Instead, it gives an intrinsic criterion, positive sufficient conditions in EDPC/n-potent settings, and a general finite-tag obstruction scheme.

The first source is Johnstone's \emph{Stone Spaces}.  Chapter VI of that monograph is devoted to ``Profiniteness and duality'', including ind-objects and pro-objects, profinite sets and algebras, Stone-type dualities, and general concrete dualities \cite[Ch.~VI]{Johnstone1982}.  The general universal-algebraic study of profinite algebras goes back at least to Banaschewski \cite{Banaschewski1972}.  In the topological-algebra literature, syntactic and principal congruences were used systematically to analyze profiniteness and standardness of Boolean topological algebras and topological prevarieties by Clark--Davey--Freese--Jackson and Clark--Davey--Jackson--Pitkethly \cite{ClarkDaveyFreeseJackson2005,ClarkDaveyJacksonPitkethly2008}.  A central theme in this area is that, in several algebraic settings, a topological algebra is profinite precisely when its underlying space is Stone.  This is classical for groups, semigroups, rings, and distributive lattices.  Numakura's compact totally disconnected semigroup and lattice results are early examples \cite{Numakura1957}; Schneider and Zumbr\"agel later gave a general explanation using affine boundedness and equicontinuity of the translation monoid \cite{SchneiderZumbragel2017}; Almeida, Goulet-Ouellet, and Kl\'ima reformulated and sharpened the same circle of ideas via syntactic congruences \cite{AlmeidaGouletOuelletKlima2023}.  Related lattice-based cases include profinite orthomodular lattices \cite{ChoeGreechie1993} and compact Hausdorff Heyting algebras \cite{BezhanishviliHarding2016}.  The latter case is exceptional in a useful way: compact Hausdorffness alone forces the topology of a topological Heyting algebra to be Stone.  For arbitrary residuated lattices, one should not expect such automatic zero-dimensionality; the exact topological condition is the Stone condition itself.

We prove that a topological residuated lattice is profinite if and only if its underlying space is Stone.  The theorem follows at the level of the semiring reduct from the semiring case of Schneider--Zumbr\"agel, but that reduct-level statement alone does not say that the residual operation, or even the meet operation, is respected by the finite quotients obtained there.  The proof given here is deliberately residuated-lattice-theoretic: the semiring reduct supplies clopen finite semiring quotients, and the $1$-classes of those quotients are clopen filters that induce finer genuine residuated-lattice congruences.  In this way the finite quotients used in the inverse limit are described internally by filters, not merely by semiring kernels.

The second source is Gr\"atzer's representation problem.  In \cite[p.~156, Problems 33 and 34]{Gratzer1971}, Gr\"atzer asked for a characterization of the poset $\mathcal P(L)$ of prime ideals of a distributive lattice $L$, and then asked the corresponding problem when $L$ has a zero, a unit, or both.  Since our residuated lattices are bounded, the bounded version is the relevant one.  Through Priestley duality \cite{Priestley1970,Priestley1972} and the theory of profinite posets \cite{Speed1972}, this problem can be read as a question about which ordered sets underlie the prime spectra of bounded distributive lattices.  Esakia restated Gr\"atzer's problem for Heyting algebras in Appendix A.5 of the 2019 English edition of \emph{Heyting Algebras: Duality Theory} \cite[Appendix A.5]{Esakia2019}.  Bezhanishvili, Bezhanishvili, Moraschini, and Stronkowski then gave a deep Heyting-algebra solution involving Esakia duality, image-finite posets, diamond systems, and the variety $\DHA$ of diamond Heyting algebras \cite{BezhanishviliEtAl2021}.

The MV-algebraic version is another important benchmark for spectrum representation.  Through Mundici's equivalence, the spectrum of an MV-algebra can be studied via the spectrum of the corresponding Abelian $\ell$-group with strong unit.  Plo\v{s}\v{c}ica and Wehrung recently solved Mundici's MV-spectrum problem up to cardinality $\aleph_1$: in the lattice-dual form, a distributive $0$-lattice of cardinality at most $\aleph_1$ is representable as $\Id_c G$, the lattice of principal $\ell$-ideals of an Abelian $\ell$-group $G$, precisely when it is completely normal and has countably based differences; the bound $\aleph_1$ is sharp \cite{PloscicaWehrung2024}.  This result illustrates that, beyond finite or strongly restricted classes, describing the relevant distributive-lattice spectra can require substantial additional representation theory.

The point of this paper is that the Heyting solution does not extend by simply replacing Heyting algebras with residuated lattices.  In a Heyting algebra the monoidal operation is meet, $x\odot y=x\wedge y$, so the Priestley spectrum carries additional Esakia structure encoding implication.  In a general residuated lattice, however, the correct bridge to distributive-lattice spectra is the \emph{reticulation}.  The reticulation $\mathcal L(A)$ of a residuated lattice $A$ is a bounded distributive lattice whose prime spectrum is naturally homeomorphic and order-isomorphic to the prime-filter spectrum of $A$ \cite{Muresan2008}.  Thus the general residuated-lattice representation problem is a problem about the Priestley spectra of the lattices in the image of the reticulation functor.

The third source is the profinite completion problem.  For an algebra $A$, its profinite completion is the inverse limit of its finite homomorphic images.  Every such completion is profinite, but not every profinite algebra is a profinite completion.  This distinction has been studied in general algebra, natural duality theory, Heyting algebras, MV-algebras, modal algebras, orthomodular lattices, and group theory \cite{Banaschewski1972,ClarkDaveyFreeseJackson2005,ClarkDaveyJacksonPitkethly2008,HardingBezhanishvili2004,BezhanishviliGehrkeMinesMorandi2006,BezhanishviliVosmaer2008,Vosmaer2009,DaveyGouveiaHaviarPriestley2011,Nganou2018,BrunsGreechieHardingRoddy1990,BridsonGrunewald2004}.  In the residuated-lattice setting, the problem is especially delicate because the Heyting fragment is idempotent whereas general residuated lattices may contain nilpotent or non-idempotent finite factors.  We do not claim a simple external list of all profinite residuated lattices that are completions.  Instead we first give an intrinsic completion criterion by comparing the profinite topology of a dense abstract subalgebra with the topology induced by the ambient profinite algebra.  This criterion is in the spirit of finite congruence control in topological algebra, but it addresses a different completion-level question: the ambient algebra is already profinite, and the issue is whether it is the full profinite completion of some dense abstract algebra.  We then prove that the known Heyting non-completion obstruction reflects into the category of residuated lattices, and we strengthen this by identifying a finite-tag obstruction mechanism.  In particular, for $T=\mathbf L_3$ we obtain a genuinely non-idempotent profinite residuated lattice, of the concrete form $\Up(X)\times\mathbf L_3$ for an image-finite poset $X$ arising from the Heyting non-completion construction, which is not the profinite completion of any residuated lattice.

The main results of the paper fall into three groups.

\emph{Stone profiniteness.}  We prove that a topological residuated lattice is profinite if and only if its underlying topological space is Stone.  More precisely, the finite quotients may be chosen internally: clopen semiring congruences of $(A,\vee,\odot,0,1)$ admit cofinal refinements by clopen residuated-lattice congruences induced by clopen filters.  This gives an intrinsic inverse-limit representation by finite residuated-lattice quotients and yields a pro-$\mathcal V$ statement for every variety $\mathcal V$ of residuated lattices.

\emph{Reticulation spectra and Gr\"atzer representability.}  For every class $\mathcal K$ of residuated lattices, we reduce the $\mathcal K$-spectrum representation problem to the Priestley representation problem for the reticulation image $\Ret(\mathcal K)$.  We then show that this reduction is optimal in the finite case: arbitrary isomorphism-closed classes of finite bounded distributive lattices occur as reticulation images of classes of finite residuated lattices.  Thus the reticulation image is not an artifact of the proof, but the exact amount of information needed for a classwise Gr\"atzer theorem.  The recent solution of the MV-spectrum problem in size $\aleph_1$ by Plo\v{s}\v{c}ica and Wehrung provides a useful comparison point: in natural infinite classes, determining the relevant distributive-lattice spectra can be a deep representation problem in its own right.  We also compute this image in several important finite and logical cases.  In particular, the reticulation of a finite residuated lattice is carried by its idempotent skeleton, and
\[
 \Ret(\Fin\MTL)=\Ret(\Fin\BL)=\Fin\GA.
\]
Thus finite posets representable by finite MTL- or BL-algebras are exactly finite root systems.  We also give a uniform reticulation theorem for $n$-potent varieties, connecting the construction with DPC/EDPC phenomena in substructural logic.  In such varieties the reticulation is no longer merely an external distributive lattice: it is the internal power skeleton $A^{[n]}=\{a^n:a\in A\}$.

\emph{Profinite completions and non-completion obstructions.}  We first give an intrinsic criterion for a profinite residuated lattice to be a profinite completion: it must contain a dense abstract subalgebra whose finite-index congruences are exactly those induced by clopen congruences of the profinite algebra.  We then prove a positive completion theorem for $n$-potent, equivalently EDPC, varieties: if a dense subalgebra has all finite-index congruences finitely generated, then the ambient profinite algebra is its profinite completion.  This shows that EDPC gives strong finite-congruence control, but also clarifies its limits.  We then prove that a profinite Heyting algebra is a profinite completion in residuated lattices exactly when it is a profinite completion in Heyting algebras.  Finally we show that non-completion is governed by a finite-tag obstruction mechanism.  If $T$ is invisible to finite Heyting quotients, then $H\times T$ can be a profinite completion only if $H$ is a Heyting profinite completion.  Taking $T=\mathbf L_3$ and specializing $H$ to the concrete algebra $\Up(X)$ from the Heyting non-completion construction yields a genuinely non-idempotent profinite residuated lattice which is not the profinite completion of any residuated lattice; more generally the same argument gives a family of such examples for non-idempotent finite tags.

The organization follows these three groups.  Section~2 recalls filters, congruences, spectra, and the standard subvarieties of residuated lattices used later.  Section~3 proves the Stone profiniteness theorem and the cofinal lifting of clopen semiring congruences.  Sections~4--7 develop reticulation images, the sharpness of the reticulation-image reduction, finite reticulation calculus, $n$-potent reticulation, the DPC/EDPC power-skeleton mechanism, and consequences for standard substructural varieties and Gr\"atzer representability.  Section~8 treats profinite completions: it first gives the intrinsic completion criterion and a positive EDPC/n-potent completion criterion under finite generation of finite-index congruences, then proves completion transfer for profinite Heyting algebras, and finally develops the finite-tag obstruction leading to genuinely non-idempotent non-completion examples.  Section~9 explains why the diamond-Heyting method is confined to the idempotent fragment and why reticulation, together with finite non-idempotent tags, is the correct mechanism beyond it.  Section~10 concludes.

\section{Preliminaries}

Throughout the paper, following the terminology originating with Ward and Dilworth \cite{WardDilworth1939}, a \emph{residuated lattice} means a bounded commutative integral residuated lattice
\[
 A=(A,\wedge,\vee,\odot,\to,0,1).
\]
Equivalently, this is the bounded commutative integral $\mathsf{FL}_{ew}$-algebraic setting in the standard substructural-logic terminology.  We fix this convention throughout and do not use the phrase ``residuated lattice'' for non-integral, non-commutative, or unbounded variants.  We denote by $\RL$ the variety of all residuated lattices in this sense.  Thus $(A,\wedge,\vee,0,1)$ is a bounded lattice, $(A,\odot,1)$ is a commutative monoid, and
\[
 x\odot y\le z \quad\Longleftrightarrow\quad x\le y\to z
\]
for all $x,y,z\in A$.  We write $x^n$ for the $n$-fold monoidal product of $x$, with $n\ge 1$.

We shall use the following standard facts:
\[
 x\le y \Longleftrightarrow x\to y=1,
 \qquad x\odot y\le x\wedge y,
 \qquad x\odot(y\vee z)=(x\odot y)\vee(x\odot z),
\]
where the last equality follows because $x\odot-$ is a left adjoint.  In particular,
\[
 x\odot 0=0.
\]
Consequently
\[
 S(A):=(A,\vee,\odot,0,1)
\]
is a semiring: addition is $\vee$, multiplication is $\odot$, $0$ is the additive identity and multiplicative zero, and $1$ is the multiplicative identity.

We shall also use the standard t-norm-based subvarieties $\MTL$, $\BL$, $\MV$, and $\GA$.  In our bounded commutative integral signature, $\MTL$ is given by prelinearity $(x\to y)\vee(y\to x)=1$; $\BL$ is $\MTL$ plus divisibility $x\wedge y=x\odot(x\to y)$; $\MV$ is the involutive part of $\BL$; and $\GA$ is the idempotent part of $\BL$, equivalently the variety of prelinear Heyting algebras.  This terminology follows the standard substructural-logic taxonomy in \cite[Secs.~2.3 and 3.5]{GalatosJipsenKowalskiOno2007}.

A \emph{topological residuated lattice} is a residuated lattice equipped with a topology for which $\wedge,\vee,\odot,\to$ are continuous.  A \emph{Stone space} means a compact Hausdorff zero-dimensional space.  A topological algebra is \emph{profinite} if it is isomorphic, as a topological algebra, to an inverse limit of finite discrete algebras.

\begin{definition}
A \emph{filter} of a residuated lattice $A$ is a non-empty subset $F\subseteq A$ such that
\begin{enumerate}
\item $x,y\in F$ implies $x\odot y\in F$;
\item $x\in F$ and $x\le y$ imply $y\in F$.
\end{enumerate}
For a filter $F$, put
\[
 \theta_F=\{(x,y)\in A^2:(x\to y)\wedge(y\to x)\in F\}.
\]
Equivalently one may use $(x\to y)\odot(y\to x)$ in place of the meet; the two definitions yield the same congruence.
\end{definition}

The following standard correspondence is one reason filters are the natural replacement for ideals in residuated lattices; see, for example, \cite[Ch.~2]{GalatosJipsenKowalskiOno2007}.

\begin{proposition}[Filter--congruence correspondence]\label{prop:filter-cong}
For every residuated lattice $A$, the assignment
\[
 F\longmapsto \theta_F
\]
is an isomorphism from the lattice of filters of $A$ to the lattice of congruences of $A$.  Its inverse sends a congruence $\theta$ to the $1$-class
\[
 [1]_\theta=\{x\in A:(x,1)\in\theta\}.
\]
\end{proposition}

A proper filter $P$ of $A$ is \emph{prime} if
\[
 x\vee y\in P\quad\Longrightarrow\quad x\in P\text{ or }y\in P.
\]
We write $\Spec_{\RL}(A)$ for the set of prime filters of $A$, ordered by inclusion.  If $D$ is a bounded distributive lattice, $\Spec(D)$ denotes the poset of prime lattice filters of $D$, again ordered by inclusion.  Gr\"atzer's original formulation is in terms of prime ideals; the filter convention used here is order-dual to the ideal convention.

\section{Stone topological residuated lattices are profinite}

We shall use the following standard compact-algebra criterion for profinite algebras.  It is included to fix notation and to emphasize that the main point in the present section is not the criterion itself, but the construction of enough clopen congruences in the full residuated-lattice signature.

\begin{lemma}[Standard profiniteness criterion]\label{lem:standard_profinite_criterion}
Let $A$ be a compact topological algebra of finite signature.  Then $A$ is profinite if and only if its clopen finite-index congruences separate points; that is, for all $a\ne b$ in $A$, there is a clopen finite-index congruence $\theta$ such that $(a,b)\notin\theta$.
\end{lemma}

\begin{proof}
If $A$ is profinite, write it as an inverse limit of finite discrete algebras.  The kernels of the coordinate projections are clopen finite-index congruences and separate points.

Conversely, let $\Theta$ be the directed family of all clopen finite-index congruences of $A$, ordered by reverse inclusion.  For each $\theta\in\Theta$, the quotient $A/\theta$ is a finite discrete algebra.  The natural homomorphisms give a continuous map
\[
 \eta:A\longrightarrow \varprojlim_{\theta\in\Theta} A/\theta,
 \qquad a\longmapsto ([a]_\theta)_\theta .
\]
The separating assumption implies that $\eta$ is injective.  Moreover, the projections of $\eta[A]$ onto each finite quotient $A/\theta$ are surjective, and the family $\Theta$ is closed under finite intersections; hence $\eta[A]$ is dense in the inverse limit.  Since $A$ is compact and the inverse limit is Hausdorff, $\eta[A]$ is compact and therefore closed.  Thus $\eta$ is onto, and $A$ is isomorphic to an inverse limit of finite discrete algebras.
\end{proof}

The criterion is standard in the theory of profinite universal algebras; see, for example, Banaschewski's pioneering account of profinite universal algebras \cite{Banaschewski1972}, the syntactic/principal-congruence approach of Clark--Davey--Freese--Jackson \cite{ClarkDaveyFreeseJackson2005}, and the later formulations in Schneider--Zumbr\"agel \cite{SchneiderZumbragel2017} and Almeida--Goulet-Ouellet--Kl\'ima \cite{AlmeidaGouletOuelletKlima2023}.  Thus Lemma~\ref{lem:standard_profinite_criterion} reduces the Stone--profinite problem to constructing enough clopen congruences.  The difficulty in the present setting is that a congruence of a reduct need not be a congruence of the full residuated-lattice algebra.

Our argument begins with the following theorem of Schneider and Zumbr\"agel for the semiring reduct.

\begin{theorem}[Schneider--Zumbr\"agel, topological semirings]\label{thm:SZ-semiring}
For any topological semiring $R=(R,+,\cdot)$, $R$ is profinite if and only if its underlying topological space is Stone.
\end{theorem}

This is Proposition~7.2 of \cite{SchneiderZumbragel2017}; it follows from their affine-boundedness theorem.  We use it specifically for the finite semiring signature
\[
S(A)=(A,\vee,\odot,0,1),
\]
where addition is \(\vee\), multiplication is \(\odot\), and the constants are the bottom element and the monoidal unit.  Thus the result applies directly to the semiring reduct of every topological residuated lattice considered here.

The next observation is elementary, but it is the first point at which the integral residuated-lattice structure enters the proof.  A congruence of the semiring reduct need not be a congruence of the full residuated lattice, since it need not preserve either $\wedge$ or $\to$.  Nevertheless, its $1$-class has the correct order-theoretic form: it is a filter.  This is the entry point for replacing semiring congruences by finer congruences in the full residuated-lattice signature.

\begin{lemma}[The $1$-class filter lemma]\label{lem:one-class-filter}
Let $A$ be a residuated lattice and let $\rho$ be a congruence of the semiring reduct $S(A)=(A,\vee,\odot,0,1)$.  Then
\[
 F=[1]_\rho=\{x\in A:(x,1)\in\rho\}
\]
is a filter of $A$.
If $A$ is topological and $\rho$ is clopen in $A^2$, then $F$ is clopen in $A$.
\end{lemma}

\begin{proof}
Clearly $1\in F$.  If $x,y\in F$, then compatibility of $\rho$ with $\odot$ gives
\[
 x\odot y\equiv_\rho 1\odot 1=1,
\]
so $x\odot y\in F$.  If $x\in F$ and $x\le y$, then $y=x\vee y$.  Hence
\[
 y=x\vee y\equiv_\rho 1\vee y=1,
\]
so $y\in F$.  Therefore $F$ is a filter.  The final assertion follows because
\[
 F=\{x\in A:(x,1)\in\rho\}
\]
is the inverse image of the clopen set $\rho$ under the continuous map $x\mapsto(x,1)$.
\end{proof}

The lemma does not yet use the residual operation.  Its role is to turn a reduct-level congruence into a clopen filter.  The residual operation then enters through the standard filter--congruence correspondence recalled in Proposition~\ref{prop:filter-cong}: a filter $F$ induces the full residuated-lattice congruence
\[
 \theta_F=\{(x,y):(x\to y)\wedge(y\to x)\in F\}.
\]
The next theorem applies this mechanism to separate arbitrary pairs of points; Proposition~\ref{prop:cofinal-lifting} below records the stronger cofinal lifting statement.

\begin{theorem}\label{thm:stone-profinite-rl}
Let $A$ be a topological residuated lattice.  Then $A$ is profinite if and only if its underlying topological space is Stone.
\end{theorem}

\begin{proof}
Every profinite algebra has underlying space an inverse limit of finite discrete spaces, hence is compact, Hausdorff, and zero-dimensional.

Conversely, assume that the underlying space of $A$ is Stone.  Then the semiring reduct
\[
 S(A)=(A,\vee,\odot,0,1)
\]
is a Stone topological semiring.  By Theorem \ref{thm:SZ-semiring}, $S(A)$ is profinite as a topological semiring.  Thus finite clopen semiring congruences separate points of $S(A)$.

Let $a\ne b$ in $A$ and set
\[
 d=(a\to b)\wedge(b\to a).
\]
Then $d\ne 1$, since $d=1$ would imply $a\to b=1$ and $b\to a=1$, hence $a\le b\le a$.  Since $S(A)$ is profinite, choose a clopen finite-index semiring congruence $\rho$ such that
\[
 (d,1)\notin\rho.
\]
Let $F=[1]_\rho$.  By Lemma \ref{lem:one-class-filter}, $F$ is a clopen filter of $A$.  Let $\theta_F$ be the corresponding residuated-lattice congruence from Proposition \ref{prop:filter-cong}.  Since the map
\[
 e:A^2\to A,
 \qquad e(x,y)=(x\to y)\wedge(y\to x)
\]
is continuous and $F$ is clopen, we have
\[
 \theta_F=e^{-1}(F),
\]
so $\theta_F$ is clopen in $A^2$.  Moreover, $(a,b)\notin\theta_F$ because $d\notin F$.  The quotient $A/\theta_F$ is finite: a clopen equivalence relation on a compact space has compact discrete quotient, hence finite quotient.

We have shown that clopen finite-index residuated-lattice congruences separate points of $A$.  By Lemma \ref{lem:standard_profinite_criterion}, $A$ is profinite.
\end{proof}

\begin{proposition}[Cofinal lifting of semiring congruences]\label{prop:cofinal-lifting}
Let $A$ be a topological residuated lattice and let
\[
 \rho\subseteq A\times A
\]
be a clopen congruence of the semiring reduct $S(A)=(A,\vee,\odot,0,1)$.  Put
\[
 F_\rho=[1]_\rho.
\]
Then $F_\rho$ is a clopen filter of $A$, and the congruence
\[
 \theta_{F_\rho}=\{(x,y)\in A^2:(x\to y)\wedge(y\to x)\in F_\rho\}
\]
is a clopen residuated-lattice congruence satisfying
\[
 \theta_{F_\rho}\subseteq \rho.
\]
Consequently, clopen semiring congruences admit cofinal refinements by clopen residuated-lattice congruences.
\end{proposition}

\begin{proof}
By Lemma \ref{lem:one-class-filter}, $F_\rho$ is a clopen filter.  Hence Proposition \ref{prop:filter-cong} gives a residuated-lattice congruence $\theta_{F_\rho}$, and it is clopen because
\[
 \theta_{F_\rho}=e^{-1}(F_\rho),\qquad
 e(x,y)=(x\to y)\wedge(y\to x),
\]
where $e$ is continuous.

It remains only to prove $\theta_{F_\rho}\subseteq\rho$.  Suppose $(x,y)\in\theta_{F_\rho}$ and set
\[
 c=(x\to y)\wedge(y\to x).
\]
Then $c\in F_\rho$, so $c\,\rho\,1$.  Hence, by compatibility of $\rho$ with the semiring multiplication,
\[
 c\odot x\,\rho\,1\odot x=x,
 \qquad
 c\odot y\,\rho\,1\odot y=y.
\]
On the other hand,
\[
 c\odot x\le y,
 \qquad
 c\odot y\le x,
\]
because $c\le x\to y$ and $c\le y\to x$.  Since $u\le v$ is equivalent to $v=u\vee v$, we obtain
\[
 y=(c\odot x)\vee y\,\rho\,x\vee y,
\]
and similarly
\[
 x=(c\odot y)\vee x\,\rho\,y\vee x=x\vee y.
\]
Thus $x\,\rho\,x\vee y\,\rho\,y$, and so $(x,y)\in\rho$.
\end{proof}

\begin{remark}[Relation with the semiring theorem]\label{rem:semiring-relation}
Proposition \ref{prop:cofinal-lifting} clarifies the precise relationship with the general semiring theorem of Schneider and Zumbr\"agel.  Their result applies to the semiring reduct $S(A)$ and yields the profiniteness of $S(A)$ whenever the carrier space is Stone.  But a finite semiring quotient of $S(A)$ need not respect the residual operation $\to$, nor the meet operation $\wedge$, and therefore need not be a quotient in the category of residuated lattices.  The additional residuated-lattice content is the cofinal lifting
\[
 \rho\quad\leadsto\quad F_\rho=[1]_\rho\quad\leadsto\quad\theta_{F_\rho}\subseteq\rho.
\]
Thus the finite quotients in Theorem \ref{thm:stone-profinite-rl} are not obtained by merely forgetting structure; they are described internally by clopen filters and form a cofinal separating system of finite residuated-lattice quotients.
\end{remark}

\begin{corollary}[Intrinsic inverse-limit representation]\label{cor:intrinsic-filter-limit}
Let $A$ be a Stone topological residuated lattice.  Let $\mathsf{ClFilt}(A)$ denote the directed family of clopen filters of $A$, ordered by reverse inclusion.  Then
\[
 A\cong\varprojlim_{F\in\mathsf{ClFilt}(A)} A/\theta_F,
\]
where each $A/\theta_F$ is a finite discrete residuated lattice.  More generally, if $\mathcal V$ is a variety of residuated lattices and $A\in\mathcal V$, then $A$ is pro-$\mathcal V$.
\end{corollary}

\begin{proof}
If $F$ is clopen, then $\theta_F=e^{-1}(F)$ is clopen, hence $A/\theta_F$ is compact and discrete, therefore finite.  Intersections of clopen filters are clopen filters, so the family is directed.  The family separates points by the proof of Theorem \ref{thm:stone-profinite-rl}: for $a\ne b$, a clopen semiring congruence separating $(a\to b)\wedge(b\to a)$ from $1$ yields a clopen filter $F$ such that $(a,b)\notin\theta_F$.  Lemma \ref{lem:standard_profinite_criterion} then gives the displayed inverse-limit representation.  If $A\in\mathcal V$, every quotient $A/\theta_F$ belongs to $\mathcal V$ because varieties are closed under homomorphic images.
\end{proof}

\begin{remark}\label{rem:method-difference}
The proof is intentionally different from applying a general profinite-algebra theorem directly to the full residuated-lattice signature.  The topological-algebra literature contains powerful general methods based on syntactic and principal congruences \cite{ClarkDaveyFreeseJackson2005,ClarkDaveyJacksonPitkethly2008}; Schneider and Zumbr\"agel characterize profiniteness by translation-monoid equicontinuity and prove the Stone criterion for affinely bounded algebras \cite{SchneiderZumbragel2017}; Almeida, Goulet-Ouellet, and Kl\'ima give a syntactic-congruence perspective \cite{AlmeidaGouletOuelletKlima2023}.  Our proof acknowledges the semiring theorem as the reduct-level input, but the decisive additional step is Proposition \ref{prop:cofinal-lifting}: finite semiring quotients are refined by finite residuated-lattice quotients produced from clopen filters.  This is the point at which the residual operation re-enters the argument.
\end{remark}

\section{Reticulation and exact reticulation images}\label{sec:reticulation-images}

We now recall reticulation and compute its image for several classes.  Besides the existence and spectral theorem used below, the functorial behaviour of reticulation is developed in \cite{Muresan2010}.

\begin{definition}[Reticulation]\label{def:reticulation}
Let $A$ be a residuated lattice.  A \emph{reticulation} of $A$ is a pair $(D,\lambda)$ where $D$ is a bounded distributive lattice and $\lambda:A\to D$ is a surjection satisfying, for all $a,b\in A$,
\begin{align*}
 \lambda(a\odot b)&=\lambda(a)\wedge\lambda(b),
 &\lambda(a\vee b)&=\lambda(a)\vee\lambda(b),\\
 \lambda(0)&=0,
 &\lambda(1)&=1,
\end{align*}
\[
 \lambda(a)\le \lambda(b)
 \quad\Longleftrightarrow\quad
 \exists n\ge 1\text{ such that }a^n\le b.
\]
We write $\mathcal L(A)$ for the reticulation, which is unique up to bounded-lattice isomorphism.
\end{definition}

Mure\c{s}an proved that every residuated lattice has a reticulation and that it can be constructed as a quotient of $A$ by the radical relation determined by powers \cite{Muresan2008}.  We shall use the following part of the theory.

\begin{theorem}[Mure\c{s}an]\label{thm:muresan-spectrum}
Let $A$ be a residuated lattice and let $(\mathcal L(A),\lambda_A)$ be its reticulation.  The inverse-image map
\[
 \lambda_A^{-1}:\Spec(\mathcal L(A))\longrightarrow\Spec_{\RL}(A),
 \qquad P\longmapsto\lambda_A^{-1}[P]
\]
is an order isomorphism.  With the usual Stone/Priestley spectral topologies on both sides, it is a homeomorphism.
\end{theorem}

\begin{definition}\label{def:ret-image}
For a class $\mathcal K$ of residuated lattices, define its \emph{reticulation image} by
\[
 \Ret(\mathcal K)=\{D\in\BDL:D\cong\mathcal L(A)\text{ for some }A\in\mathcal K\}.
\]
\end{definition}

\begin{proposition}\label{prop:products-ret}
Reticulation preserves finite direct products.  More precisely,
\[
 \mathcal L(A\times B)\cong \mathcal L(A)\times\mathcal L(B)
\]
for all residuated lattices $A$ and $B$.
\end{proposition}

\begin{proof}
Let $\lambda_A:A\to\mathcal L(A)$ and $\lambda_B:B\to\mathcal L(B)$ be the reticulation maps.  Define
\[
 \lambda:A\times B\to\mathcal L(A)\times\mathcal L(B),
 \qquad \lambda(a,b)=(\lambda_A(a),\lambda_B(b)).
\]
It is plainly onto and preserves $0,1,\vee$, and $\odot$ coordinatewise.  It remains to check the radical order condition.  We have
\begin{align*}
 \lambda(a,b)\le \lambda(c,d)
 &\Longleftrightarrow \lambda_A(a)\le\lambda_A(c)\text{ and }\lambda_B(b)\le\lambda_B(d)\\
 &\Longleftrightarrow \exists m,n\ge 1\ (a^m\le c\text{ and }b^n\le d).
\end{align*}
Since every residuated lattice is integral, powers are decreasing: $x^{k+1}\le x^k$.  Hence, if the last condition holds, then for $N=\max(m,n)$ we get
\[
 (a,b)^N=(a^N,b^N)\le(c,d).
\]
The converse is immediate by taking coordinates.  Thus $\lambda$ is a reticulation of $A\times B$.
\end{proof}

\subsection{Heyting and idempotent cases}

A residuated lattice is \emph{idempotent} if $x^2=x$ for all $x$.

\begin{proposition}[Standard idempotent case]\label{prop:idempotent-heyting}
Let \(A\) be a residuated lattice.  The following are equivalent.
\begin{enumerate}
\item \(A\) is idempotent, i.e. \(x\odot x=x\) for all \(x\in A\).
\item \(x\odot y=x\wedge y\) for all \(x,y\in A\).
\item \(A\) is a Heyting algebra, regarded as a residuated lattice by setting \(x\odot y=x\wedge y\).
\end{enumerate}
In this case the canonical reticulation map
\[
 \lambda_A:A\longrightarrow \mathcal L(A)
\]
is a bounded-lattice isomorphism.  Hence \(A\cong \mathcal L(A)\) as bounded lattices.
\end{proposition}

\begin{proof}
The equivalence of \textup{(i)}--\textup{(iii)} is the standard identification of idempotent commutative integral residuated lattices with Heyting algebras; see, for example, the basic structural discussion of residuated lattices and their idempotent subvarieties in \cite[Secs.~2.3 and 3.5]{GalatosJipsenKowalskiOno2007}.  We only record the reticulation consequence.  If \(A\) is idempotent, then \(a^m=a\) for all \(a\in A\) and all \(m\ge1\).  Hence the reticulation order condition gives
\[
 \lambda_A(a)\le \lambda_A(b)
 \quad\Longleftrightarrow\quad
 a\le b.
\]
Thus \(\lambda_A\) is order-reflecting.  Since the reticulation map is onto and preserves \(0,1\) and finite joins, it is a bounded-lattice isomorphism.
\end{proof}

\begin{remark}
Proposition~\ref{prop:idempotent-heyting} is the \(n=1\) instance of the internal power-skeleton description established below for \(n\)-potent residuated lattices.
\end{remark}

\begin{corollary}\label{cor:ret-ha}
Let $\mathcal V$ be any class of Heyting algebras, viewed as a class of residuated lattices.  Then
\[
 \Ret(\mathcal V)=
 \left\{\begin{array}{c|c}
 D\in\BDL &
 D\text{ is the bounded-lattice reduct of some }H\in\mathcal V
 \end{array}\right\}.
\]
In particular:
\begin{enumerate}
\item $\Ret(\HA)$ is the class of bounded distributive lattices admitting a relative pseudocomplement;
\item $\Ret(\BA)=\BA$;
\item $\Ret(\GA)$ is the class of bounded distributive lattices admitting a prelinear relative pseudocomplement;
\item $\Ret(\DHA)$ is the class of bounded-lattice reducts of diamond Heyting algebras.
\end{enumerate}
\end{corollary}

\begin{proof}
This is immediate from Proposition \ref{prop:idempotent-heyting}.  The descriptions of $\HA$, $\BA$, and $\GA$ use the usual equational definitions: Heyting algebras are relatively pseudocomplemented bounded distributive lattices, Boolean algebras are complemented distributive lattices, and G\"odel algebras are prelinear Heyting algebras.
\end{proof}

\subsection{Finite and chain cases}

\begin{definition}\label{def:omega-idempotent}
Let $A$ be a finite residuated lattice.  For $a\in A$, the descending sequence
\[
 a\ge a^2\ge a^3\ge\cdots
\]
eventually stabilizes.  We denote its eventual value by $a^\omega$.  Thus $a^\omega$ is idempotent and $\lambda_A(a)=\lambda_A(a^\omega)$ in the reticulation.  Let
\[
 \Idem(A)=\{e\in A:e^2=e\}.
\]
\end{definition}

\begin{proposition}[Finite reticulation calculus]\label{prop:finite-idempotent-skeleton}
Let $A$ be a finite residuated lattice.  Then $\mathcal L(A)$ is isomorphic to the bounded distributive lattice carried by $\Idem(A)$ with operations
\[
 e\wedge_{\omega} f=e\odot f,
 \qquad
 e\vee_{\omega} f=(e\vee f)^\omega,
 \qquad
 0_{\omega}=0^\omega,
 \qquad
 1_{\omega}=1.
\]
Under this isomorphism, $\lambda_A(a)$ corresponds to $a^\omega$.
\end{proposition}

\begin{proof}
Let $\lambda:A\to\mathcal L(A)$ be the reticulation map.  Since $A$ is finite, $\lambda(a)=\lambda(a^\omega)$, so every element of $\mathcal L(A)$ is represented by an idempotent.  If $e,f\in\Idem(A)$ and $\lambda(e)=\lambda(f)$, then the defining order condition for reticulation gives $e\le f$ and $f\le e$, because powers of idempotents are themselves.  Thus $e=f$.  Hence $\lambda$ restricts to a bijection
\[
 \Idem(A)\longrightarrow\mathcal L(A).
\]
Transporting the lattice operations of $\mathcal L(A)$ along this bijection gives
\[
 \lambda(e\odot f)=\lambda(e)\wedge\lambda(f),
 \qquad
 \lambda((e\vee f)^\omega)=\lambda(e\vee f)=\lambda(e)\vee\lambda(f).
\]
The constants are transported similarly.  Since $\mathcal L(A)$ is a bounded distributive lattice, the operations transported along this bijection make $\Idem(A)$ into a bounded distributive lattice isomorphic to $\mathcal L(A)$.
\end{proof}

\begin{remark}\label{rem:finite-filters-idempotents}
The same phenomenon can be seen from filters.  If $F$ is a filter of a finite residuated lattice, then the product of all elements of $F$ is the least element of $F$ and is idempotent; hence every filter is principal of the form $\uparrow e$ with $e\in\Idem(A)$.  This is the finite form of the filter--congruence structure emphasized in Galatos--Jipsen--Kowalski--Ono \cite[Theorem~3.47]{GalatosJipsenKowalskiOno2007}.
\end{remark}

\begin{proposition}\label{prop:finite-image}
Let $\Fin\RL$ be the class of finite residuated lattices and $\Fin\BDL$ the class of finite bounded distributive lattices.  Then
\[
 \Ret(\Fin\RL)=\Fin\BDL.
\]
\end{proposition}

\begin{proof}
If $A$ is finite, then $\mathcal L(A)$ is a finite bounded distributive lattice.  Conversely, let $D$ be a finite bounded distributive lattice.  For $a,b\in D$, define
\[
 a\Rightarrow b=\bigvee\{c\in D:a\wedge c\le b\}.
\]
The join exists because $D$ is finite, and distributivity implies that it is the relative pseudocomplement of $a$ relative to $b$.  Hence $D$ is a finite Heyting algebra.  By Proposition \ref{prop:idempotent-heyting}, $\mathcal L(D)\cong D$.
\end{proof}

\begin{proposition}\label{prop:chain-image}
Let $\Ch\RL$ be the class of residuated lattices whose lattice order is a chain, and let $\Ch\BDL$ be the class of bounded distributive lattices whose order is a chain.  Then
\[
 \Ret(\Ch\RL)=\Ch\BDL.
\]
\end{proposition}

\begin{proof}
If $A\in\Ch\RL$, then for any $a,b\in A$ we have $a\le b$ or $b\le a$, whence $\lambda(a)\le\lambda(b)$ or $\lambda(b)\le\lambda(a)$.  Since $\lambda$ is onto, $\mathcal L(A)$ is a chain.

Conversely, let $D$ be a bounded distributive chain.  Define
\[
 a\Rightarrow b=\begin{cases}
 1,&a\le b,\\
 b,&a>b.
 \end{cases}
\]
Then $a\Rightarrow b$ is the greatest $c$ such that $a\wedge c\le b$.  Thus $D$ is a Heyting chain, and Proposition \ref{prop:idempotent-heyting} gives $\mathcal L(D)\cong D$.
\end{proof}

\begin{corollary}\label{cor:finite-posets}
Every finite poset is representable as the prime-filter spectrum of a finite residuated lattice.
\end{corollary}

\begin{proof}
Let $P$ be a finite poset and let $D=\Up(P)$ be the finite bounded distributive lattice of upsets of $P$.  For $p\in P$, put
\[
 F_p=\{U\in\Up(P):p\in U\}.
\]
Then $p\mapsto F_p$ is an order isomorphism from $P$ to $\Spec(D)$.  By Proposition \ref{prop:finite-image}, $D$ is the reticulation of a finite residuated lattice, indeed of the finite Heyting algebra $D$ itself.
\end{proof}

\subsection{Finite and finite-valued MV-cases}

We use the standard notation $\MV$ for MV-algebras, written in the equivalent residuated-lattice signature.  The $n{+}1$-element MV-chain is denoted by $\mathbf L_{n+1}$.  The variety generated by $\mathbf L_{n+1}$ is denoted by $\MV_n$; these are the usual finite-valued MV-varieties.  We recall the standard facts that finite MV-algebras are finite products of finite MV-chains and that the linearly ordered members of $\MV_n$ are finite chains; see \cite{CignoliDOttavianoMundici2000}.

The MV-case has its own profinite and dual-space literature.  Di Nola and Grigolia studied pro-finite MV-spaces \cite{DiNolaGrigolia2004}, while Nganou characterized profinite MV-algebras and related them to multisets \cite{Nganou2015}; he also compared profinite and MacNeille completions for MV-algebras \cite{Nganou2018}.  Our use of finite MV-chains below is different but complementary: finite MV-chains act as test objects for reticulation and, later, the three-element chain $\mathbf L_3$ serves as a minimal non-idempotent tag in the completion-reflection argument.

\begin{lemma}\label{lem:finite-mv-chain-ret}
For every finite MV-chain $\mathbf L_{n+1}$ with $n\ge 1$,
\[
 \mathcal L(\mathbf L_{n+1})\cong \mathbf 2.
\]
\end{lemma}

\begin{proof}
Let $a<1$ in $\mathbf L_{n+1}$.  In a finite MV-chain every proper element is nilpotent with respect to $\odot$: there is $k\ge 1$ such that $a^k=0$.  Therefore
\[
 \lambda(a)\le\lambda(0),
\]
and so $\lambda(a)=0$.  Since $\lambda(1)=1$, the image has exactly two elements.
\end{proof}

\begin{proposition}\label{prop:finite-mv-ret}
Let $\Fin\MV$ be the class of finite MV-algebras.  Then
\[
 \Ret(\Fin\MV)=\Fin\BA.
\]
\end{proposition}

\begin{proof}
Every finite MV-algebra is a finite direct product of finite MV-chains.  By Lemma \ref{lem:finite-mv-chain-ret} and Proposition \ref{prop:products-ret}, its reticulation is a finite product of copies of $\mathbf 2$, hence a finite Boolean algebra.  Conversely, every finite Boolean algebra is a finite MV-algebra and, being idempotent, is isomorphic to its reticulation.
\end{proof}

\begin{proposition}\label{prop:finite-valued-mv-ret}
For every $n\ge 1$,
\[
 \Ret(\MV_n)=\BA.
\]
\end{proposition}

\begin{proof}
Let $A\in\MV_n$.  We show that $\mathcal L(A)$ is Boolean.  By Theorem \ref{thm:muresan-spectrum}, $\Spec(\mathcal L(A))$ is order-isomorphic to $\Spec_{\RL}(A)$.  If $P\in\Spec_{\RL}(A)$, then $A/P$ is a linearly ordered member of $\MV_n$.  By the standard representation theory of finite-valued MV-varieties, every linearly ordered member of $\MV_n$ is one of the finite \L{}ukasiewicz chains occurring in that finite-valued variety; see \cite[Ch.~3]{CignoliDOttavianoMundici2000}.  Hence $A/P$ is a finite MV-chain, and therefore simple.  Therefore $P$ is maximal.  Thus all prime filters of $A$ are maximal, so the Priestley order of $\Spec(\mathcal L(A))$ is equality.  A bounded distributive lattice whose Priestley dual has equality order is Boolean.  Hence $\mathcal L(A)\in\BA$.

Conversely, every Boolean algebra is an MV-algebra belonging to every $\MV_n$, and its reticulation is itself by Proposition \ref{prop:idempotent-heyting}.  Therefore $\BA\subseteq\Ret(\MV_n)$.
\end{proof}

\begin{remark}\label{rem:ret-images-meaning}
The preceding computations should be read as the residuated-lattice analogue of computing Esakia-dual conditions for Heyting varieties.  For arbitrary residuated lattices the relevant invariant is not an Esakia space attached directly to $A$, but the Priestley space of $\mathcal L(A)$.
\end{remark}

\subsection{Prelinear varieties and finite root systems}\label{subsec:prelinear-root}

We next use the prelinear varieties introduced in Section~2.  The additional hypothesis of prelinearity has a precise order-theoretic effect on spectra: it forces principal upsets in the prime-filter order to be chains.  In the finite case this necessary condition will turn out to be sufficient.

\begin{definition}
A poset $P$ is called a \emph{root system} if, for every $p\in P$, the principal upset
\[
 \uparrow p=\{q\in P:p\le q\}
\]
is a chain.  Let $\Root\BDL$ be the class of bounded distributive lattices whose Priestley spectra are root systems, and let $\Fin\Root\BDL$ be its finite part.
\end{definition}

\begin{lemma}\label{lem:prelinear-root-necessary}
If $A$ is an $\MTL$-algebra, then $\Spec_{\RL}(A)$ is a root system.  Consequently
\[
 \Ret(\MTL)\subseteq \Root\BDL,
 \qquad
 \Ret(\BL)\subseteq \Root\BDL.
\]
\end{lemma}

\begin{proof}
Let $P\in\Spec_{\RL}(A)$.  We prove that the set of prime filters above $P$ is linearly ordered.  If $Q,R\in\Spec_{\RL}(A)$ and $P\subseteq Q,R$, then $Q/P$ and $R/P$ are filters of the quotient $A/P$.  We first observe that $A/P$ is a chain.  Indeed, for any $x,y\in A$, prelinearity gives
\[
 (x\to y)\vee(y\to x)=1\in P.
\]
Since $P$ is prime, either $x\to y\in P$ or $y\to x\in P$.  Thus either $[x]\le [y]$ or $[y]\le[x]$ in $A/P$.  Hence $A/P$ is linearly ordered.  Filters of a chain are linearly ordered by inclusion, so $Q/P$ and $R/P$ are comparable.  Therefore $Q$ and $R$ are comparable.

The statement about reticulations follows from Theorem \ref{thm:muresan-spectrum}, which gives an order isomorphism
\[
  \Spec(\mathcal L(A))\cong\Spec_{\RL}(A).
\]
Since $\BL\subseteq\MTL$, the same necessary condition holds for $\BL$.
\end{proof}

For finite algebras the necessary root-system condition is also sufficient.  This gives a sharp finite Gr\"atzer theorem for the principal fuzzy-logic varieties $\MTL$ and $\BL$.

\begin{lemma}\label{lem:finite-godel-forest}
Let $P$ be a finite poset.  The finite Heyting algebra $\Up(P)$ is a G\"odel algebra if and only if $P$ is a root system.
\end{lemma}

\begin{proof}
For upsets $U,V\subseteq P$, the Heyting implication in $\Up(P)$ is given by
\[
 U\Rightarrow V=\{p\in P: \uparrow p\cap U\subseteq V\}.
\]
Assume first that $P$ is a root system.  For every $p\in P$, the upsets $U\cap\uparrow p$ and $V\cap\uparrow p$ of the chain $\uparrow p$ are comparable by inclusion.  Hence either $p\in U\Rightarrow V$ or $p\in V\Rightarrow U$.  Since this holds for every $p$,
\[
 (U\Rightarrow V)\cup(V\Rightarrow U)=P,
\]
which is prelinearity.

Conversely, suppose $P$ is not a root system.  Then there exist $p\le q,r$ with $q$ and $r$ incomparable.  Put $U=\uparrow q$ and $V=\uparrow r$.  Then $p\notin U\Rightarrow V$, because $q\in\uparrow p\cap U$ but $q\notin V$, and similarly $p\notin V\Rightarrow U$.  Therefore prelinearity fails in $\Up(P)$.
\end{proof}

\begin{theorem}[Finite prelinear reticulation theorem]\label{thm:finite-mtl-bl-ret}
The following reticulation images coincide:
\[
 \Ret(\Fin\GA)=\Ret(\Fin\BL)=\Ret(\Fin\MTL)=\Fin\Root\BDL.
\]
Consequently, a finite poset is representable as the prime-filter spectrum of a finite $\MTL$-algebra, equivalently of a finite $\BL$-algebra, if and only if it is a root system.
\end{theorem}

\begin{proof}
The inclusions
\[
 \Ret(\Fin\GA)\subseteq\Ret(\Fin\BL)\subseteq\Ret(\Fin\MTL)
\]
are immediate from $\GA\subseteq\BL\subseteq\MTL$.  By Lemma \ref{lem:prelinear-root-necessary}, every lattice in $\Ret(\Fin\MTL)$ has finite root-system spectrum, so
\[
 \Ret(\Fin\MTL)\subseteq\Fin\Root\BDL.
\]

Conversely, let $D\in\Fin\Root\BDL$ and let $P=\Spec(D)$.  Since $D$ is finite, Priestley duality identifies $D$ with $\Up(P)$.  By Lemma \ref{lem:finite-godel-forest}, $\Up(P)$ is a finite G\"odel algebra.  Since G\"odel algebras are idempotent, Proposition \ref{prop:idempotent-heyting} gives
\[
 \mathcal L(\Up(P))\cong\Up(P)\cong D.
\]
Thus $D\in\Ret(\Fin\GA)$.

The final statement follows from Theorem \ref{thm:gratz-rl}: finite $\MTL$-representability or finite $\BL$-representability of a finite poset is exactly being the spectrum of a member of the displayed common reticulation image, i.e. exactly being a finite root system.
\end{proof}

\begin{remark}\label{rem:prelinear-versus-heyting}
This theorem explains how the Heyting/Esakia part and the fuzzy-logic part interact.  In arbitrary finite residuated lattices every finite poset occurs, by Corollary \ref{cor:finite-posets}.  Under prelinearity, however, the finite spectra collapse exactly to finite root systems.  Thus $\MTL$ and $\BL$ impose a real order-theoretic restriction, but it is not the diamond-system condition from the Heyting completion theorem; it is the root-system condition forced by prelinearity.
\end{remark}

\section{Power-bounded reticulation, EDPC, and finite congruence control}
\label{sec:npotent}

The reticulation relation
\[
 \lambda(a)\le\lambda(b)\quad\Longleftrightarrow\quad \exists k\ge 1\; a^k\le b
\]
is existential in general: the exponent may depend on the pair $(a,b)$.  In many of the central substructural varieties, however, powers stabilize uniformly.  This is exactly the situation in the $n$-potent varieties.  The monograph of Galatos, Jipsen, Kowalski, and Ono places this condition in the general structure theory of residuated lattices: in the commutative integral setting, $n$-potency is the algebraic form of weak contraction, and inside $\mathsf{FL}_{ew}$ it is equivalent to DPC and EDPC for varieties \cite[p.~200, Cor.~3.56; p.~464, Thm.~11.2]{GalatosJipsenKowalskiOno2007}.  Since our convention identifies residuated lattices with bounded commutative integral $\mathsf{FL}_{ew}$-algebras, these results apply directly to the varieties considered in this paper.

We use DPC and EDPC only as finite-congruence control principles.  A variety has \emph{DPC} if membership in a principal congruence $\operatorname{Cg}(a,b)$ is uniformly first-order definable, and it has \emph{EDPC} if this definition can be made equational.  EDPC implies strong control of principal congruences and, in the residuated-lattice setting, is tied to $n$-potency.  This should not be confused with the profinite completion property: even Heyting algebras, which are idempotent and hence fall under the EDPC/n-potent regime, have profinite members that are not profinite completions.  The role of EDPC here is subtler: it makes reticulation internal and supplies positive completion criteria when all finite-index congruences are finitely generated.

The ambient congruence theory of residuated lattices is already unusually tractable.  The lattice reduct gives congruence distributivity, and the filter--congruence correspondence makes quotients and congruence generation concrete.  Standard congruence-extension results for commutative residuated-lattice varieties, together with the DPC/EDPC characterization in $\mathsf{FL}_{ew}$, show that these algebras have strong algebraic congruence control \cite[Chs.~3 and 11]{GalatosJipsenKowalskiOno2007}.  Our completion results below show that such algebraic control is not, by itself, a profinite-completion theorem: the remaining issue is topological, namely whether finite-index congruences on a dense subalgebra are induced by clopen congruences of the ambient profinite algebra.

\begin{definition}
For $n\ge 1$, a residuated lattice $A$ is called \emph{$n$-potent} if
\[
 x^{n+1}=x^n
\]
for all $x\in A$.
\end{definition}

\begin{theorem}[Uniform reticulation in $n$-potent algebras]
\label{thm:npotent-reticulation}
Let $A$ be an $n$-potent residuated lattice and let $\lambda:A\to\mathcal L(A)$ be its reticulation map.  Then for all $a,b\in A$,
\[
 \lambda(a)\le\lambda(b)
 \quad\Longleftrightarrow\quad
 a^n\le b.
\]
Consequently,
\[
 \lambda(a)=\lambda(b)
 \quad\Longleftrightarrow\quad
 a^n\le b\text{ and }b^n\le a.
\]
Moreover, if $S\subseteq A$ and $G_S$ is the lattice filter of $\mathcal L(A)$ generated by $\lambda[S]$, then
\[
 \lambda^{-1}(G_S)
 =
 \left\{x\in A:
 \exists s_1,\ldots,s_m\in S\text{ such that }(s_1\odot\cdots\odot s_m)^n\le x
 \right\}.
\]
In particular, for a finite set $S=\{s_1,
\ldots,s_m\}$,
\[
 \lambda^{-1}(G_S)=\uparrow (s_1\odot\cdots\odot s_m)^n.
\]
\end{theorem}

\begin{proof}
If $\lambda(a)\le\lambda(b)$, then by the definition of reticulation there is $k\ge1$ such that $a^k\le b$.  If $k\le n$, then integrality gives
\[
 a^n\le a^k\le b.
\]
If $k>n$, then $n$-potency gives $a^k=a^n$, and again $a^n\le b$.  The converse is immediate from the definition, taking the exponent $n$.

The characterization of equality follows by applying the order characterization in both directions.  For the assertion about generated lattice filters, recall that the filter of a bounded distributive lattice generated by a set $T$ consists of all elements above a finite meet of elements of $T$.  Since
\[
 \lambda(s_1)\wedge\cdots\wedge\lambda(s_m)
 =\lambda(s_1\odot\cdots\odot s_m),
\]
we have $\lambda(x)\in G_S$ exactly when, for some finite $s_1,
\ldots,s_m\in S$,
\[
 \lambda(s_1\odot\cdots\odot s_m)\le\lambda(x).
\]
By the first part this is equivalent to
\[
 (s_1\odot\cdots\odot s_m)^n\le x.
\]
If $S$ is finite, the finite meet of all elements of $\lambda[S]$ generates the same filter, giving the displayed principal description.
\end{proof}

\begin{lemma}[Power expansion estimate]\label{lem:power-expansion-estimate}
In every commutative integral residuated lattice, and for all \(x,y\) and integers \(r,s\ge1\),
\[
 (x\vee y)^{rs}\le x^r\vee y^s.
\]
\end{lemma}

\begin{proof}
Since multiplication distributes over finite joins in each coordinate, \((x\vee y)^{rs}\) is the join of all monomials of length \(rs\) in \(x\) and \(y\).  By commutativity, such a monomial is of the form \(x^i\odot y^{rs-i}\).  For every \(0\le i\le rs\), either \(i\ge r\) or \(rs-i\ge s\).  In the first case, integrality gives \(x^i\odot y^{rs-i}\le x^i\le x^r\); in the second case it gives \(x^i\odot y^{rs-i}\le y^{rs-i}\le y^s\).  Hence every monomial lies below \(x^r\vee y^s\), and so does their join.
\end{proof}

\begin{theorem}[$n$-potent skeleton]
\label{thm:npotent-skeleton}
Let $A$ be an $n$-potent residuated lattice.  Put
\[
 A^{[n]}=\{a^n:a\in A\}.
\]
For $u=a^n$ and $v=b^n$ define
\[
 u\vee_{[n]}v=u\vee v,
 \qquad
 u\wedge_{[n]}v=u\odot v.
\]
Then
\[
 A^{[n]}=(A^{[n]},\wedge_{[n]},\vee_{[n]},0,1)
\]
is a bounded distributive lattice, and the map
\[
 \sigma_n:A\to A^{[n]},\qquad a\mapsto a^n,
\]
is a reticulation of $A$.  Hence
\[
 \mathcal L(A)\cong A^{[n]}.
\]
\end{theorem}

\begin{proof}
We first show that $A^{[n]}$ is closed under joins.  Monotonicity gives
\[
 a^n\vee b^n\le (a\vee b)^n.
\]
Conversely, Lemma~\ref{lem:power-expansion-estimate}, applied with $r=s=n$, gives
\[
 (a\vee b)^{n^2}\le a^n\vee b^n.
\]
Using $n$-potency, $(a\vee b)^{n^2}=(a\vee b)^n$, and therefore $(a\vee b)^n=a^n\vee b^n$.

The set $A^{[n]}$ is closed under fusion because
\[
 a^n\odot b^n=(a\odot b)^n.
\]
Every element of $A^{[n]}$ is idempotent for fusion, since
\[
 a^n\odot a^n=a^{2n}=a^n.
\]
For $u,v\in A^{[n]}$, fusion distributes over $\vee$ and satisfies absorption:
\[
 u\odot(u\vee v)=(u\odot u)\vee(u\odot v)=u\vee(u\odot v)=u,
\]
because $u\odot v\le u$.  Thus fusion is the meet operation for the join order on $A^{[n]}$, and distributivity follows from distributivity of fusion over joins.  Hence $A^{[n]}$ is a bounded distributive lattice.

The map $\sigma_n$ is onto by definition, sends $0$ to $0$ and $1$ to $1$, preserves joins by the first paragraph, and sends fusion to meet by
\[
 \sigma_n(a\odot b)=(a\odot b)^n=a^n\odot b^n
 =\sigma_n(a)\wedge_{[n]}\sigma_n(b).
\]
The reticulation order condition is exactly the first assertion of Theorem \ref{thm:npotent-reticulation}, because
\[
 \sigma_n(a)\le \sigma_n(b)\text{ in }A^{[n]}
 \quad\Longleftrightarrow\quad
 a^n\le b^n
 \quad\Longleftrightarrow\quad
 a^n\le b,
\]
where the last equivalence uses $b^n\le b$ and $(a^n)^n=a^n$.  Thus $\sigma_n$ is a reticulation of $A$.
\end{proof}

\begin{corollary}\label{cor:npotent-image-skeleton}
For every class $\mathcal K$ of $n$-potent residuated lattices,
\[
 \Ret(\mathcal K)=\{A^{[n]}:A\in\mathcal K\},
\]
where each $A^{[n]}$ is equipped with the bounded distributive lattice structure of Theorem \ref{thm:npotent-skeleton}.  Thus in $n$-potent classes reticulation is an internal power-skeleton construction.
\end{corollary}

\begin{corollary}[EDPC classes have bounded reticulation complexity]
\label{cor:edpc-reticulation}
Let $\mathcal V$ be a subvariety of $\mathsf{FL}_{ew}$ whose residuated-lattice reducts are bounded, commutative and integral.  If $\mathcal V$ has EDPC, equivalently DPC, then there is an integer $n\ge1$ such that for every $A\in\mathcal V$ the reticulation order of $\mathcal L(A)$ is represented by the single inequality
\[
 \lambda(a)\le\lambda(b)
 \quad\Longleftrightarrow\quad
 a^n\le b.
\]
Thus, in EDPC substructural varieties, reticulation is not only functorial but uniformly term-controlled.
\end{corollary}

\begin{proof}
Galatos--Jipsen--Kowalski--Ono prove that, inside $\mathsf{FL}_{ew}$, a variety has DPC iff it has EDPC iff it is contained in an $n$-potent variety for some $n$ \cite[p.~200, Cor.~3.56; p.~464, Thm.~11.2]{GalatosJipsenKowalskiOno2007}.  Apply Theorem \ref{thm:npotent-reticulation} uniformly to all members of $\mathcal V$.
\end{proof}

\begin{remark}
This gives a second reason why the residuated-lattice approach differs from the Heyting one.  In Heyting algebras the exponent is $n=1$, so reticulation is the identity.  In MV-, BL-, and MTL-type situations the exponent is often not $1$ and the reticulation records only the stabilized radical order.  The power skeleton $A^{[n]}$ is exactly the part that survives at the distributive-lattice spectral level.
\end{remark}

\begin{corollary}[EDPC gives internal reticulation skeletons]
\label{cor:edpc-internal-skeleton}
Let $\mathcal V$ be a bounded commutative integral residuated-lattice variety with EDPC.  Then there is an integer $n\ge1$ such that, for every $A\in\mathcal V$,
\[
 \mathcal L(A)\cong A^{[n]}=\{a^n:a\in A\}
\]
as bounded distributive lattices, with join inherited from $A$ and meet given by fusion on $A^{[n]}$.  Consequently, at the level of prime spectra, the Gr\"atzer representation problem for $\mathcal V$ is controlled by the Priestley spectra of the internal power skeletons $A^{[n]}$.
\end{corollary}

\begin{proof}
By the DPC/EDPC theorem for subvarieties of $\mathsf{FL}_{ew}$, there is $n$ such that $\mathcal V$ is $n$-potent.  The isomorphism $\mathcal L(A)\cong A^{[n]}$ is Theorem~\ref{thm:npotent-skeleton}.  The statement about prime spectra follows from the reticulation spectrum theorem, Theorem~\ref{thm:muresan-spectrum}.
\end{proof}

\section{Applications to standard substructural varieties}
\label{sec:standard-varieties}

The notation for the standard subvarieties was fixed in Section~2.  We use this section only to record consequences of the preceding theorems for those familiar classes, without repeating the taxonomy.  The finite-embeddability results of Galatos--Jipsen--Kowalski--Ono show that several of these important varieties, including $\HA$, $\MTL$, $\BL$, and $\MV$, have FEP \cite[Sec.~6.5]{GalatosJipsenKowalskiOno2007}.

This section is not used in the proof of the Stone profiniteness theorem.  Its purpose is to compare that topological finite-quotient theorem with the algebraic finite-embeddability properties of the standard substructural varieties.  Proposition~\ref{prop:fep-free-resfinite} and Corollary~\ref{cor:fep-free-main} record algebraic residual-finiteness consequences of FEP for free algebras.  Corollary~\ref{cor:standard-varieties}, by contrast, is a direct application of the first main problem of the paper: it says that a Stone topological algebra in any of the standard subvarieties is already a profinite algebra in that same subvariety.  Thus FEP and Stone profiniteness are related finite-approximation phenomena, but they enter the paper at different levels: FEP is algebraic, whereas the Stone theorem is topological and uses clopen filter quotients.

\begin{proposition}\label{prop:fep-free-resfinite}
Let $\mathcal V$ be a variety of residuated lattices with FEP.  Then $\mathcal V$ is generated, as a variety, by its finite members.  Consequently every finitely generated free algebra $F_{\mathcal V}(n)$ is residually finite, and its canonical map into its profinite completion is injective with dense image.
\end{proposition}

\begin{proof}
Suppose an identity $s\approx t$ fails in some $A\in\mathcal V$ under an assignment $\bar a$.  Take the finite partial subalgebra of $A$ consisting of the values of all subterms of $s$ and $t$ under $\bar a$, with operations defined where the relevant subterm calculations require them.  By FEP, this finite partial algebra embeds into a finite member $B$ of $\mathcal V$.  The images of the chosen generators in $B$ still separate the values of $s$ and $t$, so $s\approx t$ fails in a finite member of $\mathcal V$.  Thus every identity true in all finite members of $\mathcal V$ is true in $\mathcal V$.

If two elements of $F_{\mathcal V}(n)$ are distinct, the corresponding identity in $n$ variables is not valid in $\mathcal V$.  Since $\mathcal V$ is generated by its finite members, the identity fails in a finite member of $\mathcal V$, and the universal property of the free algebra gives a finite quotient separating the two elements.  This is residual finiteness.  The statement about the profinite completion is the standard inverse-limit property of residually finite algebras.
\end{proof}

\begin{corollary}\label{cor:fep-free-main}
For each of $\HA,\MTL,\BL$, and $\MV$, every finitely generated free algebra is residually finite and embeds densely into its profinite completion.
\end{corollary}

\begin{proof}
The finite embeddability theorem recorded in \cite[Theorem~6.51 and Corollary~6.52]{GalatosJipsenKowalskiOno2007} covers Heyting algebras and the standard $\MTL$-varieties, while \cite[Theorems~6.54 and 6.55]{GalatosJipsenKowalskiOno2007} gives FEP for $\BL$ and $\MV$.  Apply Proposition \ref{prop:fep-free-resfinite}.
\end{proof}

\begin{corollary}[Stone equals profinite in the standard varieties]
\label{cor:standard-varieties}
Let $\mathcal V$ be any subvariety of bounded commutative integral residuated lattices, in particular any of the standard varieties
\[
 \HA,
 \GA,
 \MTL,
 \BL,
 \MV,
 \MV_n,
 \BA.
\]
If $A\in\mathcal V$ is a topological algebra whose underlying space is Stone, then
\[
 A\cong\varprojlim_{F\in\mathsf{ClFilt}(A)} A/\theta_F
\]
as a topological algebra in $\mathcal V$.  Each quotient $A/\theta_F$ is finite and belongs to $\mathcal V$.
\end{corollary}

\begin{proof}
This is Corollary \ref{cor:intrinsic-filter-limit}, using the fact that varieties are closed under homomorphic images.
\end{proof}

\begin{corollary}[Finite-valued MV-collapse]
\label{cor:mv-collapse-representability}
For each finite-valued MV-variety $\MV_n$, the $\MV_n$-representable posets are exactly the Priestley spectra of Boolean algebras, i.e. Stone spaces with equality order.  Consequently, finite-valued many-valued structure is invisible to the ordered reticulation spectrum: it collapses to Boolean reticulation.
\end{corollary}

\begin{proof}
This is Theorem \ref{thm:gratz-rl} together with Proposition \ref{prop:finite-valued-mv-ret}.
\end{proof}

\begin{corollary}[Finite prelinear collapse]\label{cor:finite-prelinear-collapse}
The finite spectra arising from finite MTL-algebras and finite BL-algebras are exactly the finite G\"odel spectra.  Thus finite prelinear many-valued multiplication creates no new finite Gr\"atzer spectra beyond the idempotent prelinear, i.e. G\"odel, case.
\end{corollary}

\begin{proof}
This is Theorem \ref{thm:gratz-rl} together with Theorem \ref{thm:finite-mtl-bl-ret}.
\end{proof}

\begin{remark}[A sharp boundary]
The previous corollary should not be confused with the Heyting case.  In $\HA$ and its subvarieties, reticulation is faithful by Proposition \ref{prop:idempotent-heyting}.  In finite-valued $\MV$-varieties, reticulation is highly non-faithful: non-Boolean finite chains still have Boolean reticulation.  This explains why Esakia/DHA methods see exactly the idempotent part, while reticulation is the correct invariant for the broader residuated-lattice setting.
\end{remark}

\section{The Gr\"atzer representation problem for residuated lattices}\label{sec:gratz}

We now formulate the residuated-lattice version of Gr\"atzer's problem.

\begin{definition}\label{def:representable}
Let $\mathcal K$ be a class of residuated lattices.  A poset $X$ is called \emph{$\mathcal K$-representable} if
\[
 X\cong \Spec_{\RL}(A)
\]
as posets for some $A\in\mathcal K$.  For the full class $\mathcal K=\RL$ of residuated lattices, where $\RL$ was fixed in Section~2, we simply say \emph{residuated-lattice representable}.
\end{definition}

We shall use Priestley duality in the following concrete form.  Every bounded distributive lattice $D$ is represented by the ordered Stone space $\Spec(D)$ of prime filters, and $D$ is recovered as the bounded lattice of clopen upsets of this space.  Conversely, if $(X,\tau,\le)$ is a Priestley space, then $\ClopUp(X,\tau)$ is a bounded distributive lattice whose prime-filter spectrum is order-homeomorphic to $X$ \cite{Priestley1970,Priestley1972}.  In the finite case this reduces to the familiar identification of a finite poset $P$ with the prime-filter spectrum of the up-set lattice $\Up(P)$.

\begin{theorem}[Reticulation form of Gr\"atzer representability]\label{thm:gratz-rl}
Let $\mathcal K$ be any class of residuated lattices and let $X$ be a poset.  The following are equivalent.
\begin{enumerate}
\item $X$ is $\mathcal K$-representable.
\item $X\cong\Spec(D)$ for some $D\in\Ret(\mathcal K)$.
\item There is a Priestley topology $\tau$ on $X$ such that
\[
 \ClopUp(X,\tau)\in\Ret(\mathcal K),
\]
where $\ClopUp(X,\tau)$ is the bounded distributive lattice of clopen upsets of the Priestley space $(X,\tau,\le)$.
\end{enumerate}
\end{theorem}

\begin{proof}
The equivalence of (i) and (ii) is exactly Theorem \ref{thm:muresan-spectrum}.  If $X\cong\Spec_{\RL}(A)$ with $A\in\mathcal K$, then
\[
 X\cong\Spec(\mathcal L(A))
\]
and $\mathcal L(A)\in\Ret(\mathcal K)$.  Conversely, if $X\cong\Spec(D)$ with $D\cong\mathcal L(A)$ and $A\in\mathcal K$, then Theorem \ref{thm:muresan-spectrum} yields
\[
 X\cong\Spec(\mathcal L(A))\cong\Spec_{\RL}(A).
\]

The equivalence of (ii) and (iii) is Priestley duality \cite{Priestley1970,Priestley1972}: bounded distributive lattices are dually equivalent to Priestley spaces, and a bounded distributive lattice $D$ is recovered from its Priestley space as the lattice of clopen upsets.
\end{proof}

\begin{corollary}[EDPC/n-potent Gr\"atzer spectra]
\label{cor:edpc-gratz-spectra}
Let $\mathcal V$ be a variety of residuated lattices contained in an $n$-potent variety.  Then a poset $X$ is $\mathcal V$-representable if and only if
\[
 X\cong\Spec(A^{[n]})
\]
for some $A\in\mathcal V$, where $A^{[n]}$ is the power skeleton of Theorem~\ref{thm:npotent-skeleton}.  In particular, if $\mathcal V$ has EDPC, then such an $n$ exists and the $\mathcal V$-version of Gr\"atzer's problem is governed by the internal skeletons $A^{[n]}$, rather than by an external Esakia-style structure.
\end{corollary}

\begin{proof}
By Theorem~\ref{thm:gratz-rl}, $X$ is $\mathcal V$-representable iff it is the Priestley spectrum of a lattice in $\Ret(\mathcal V)$.  If $\mathcal V$ is $n$-potent, Corollary~\ref{cor:npotent-image-skeleton} gives
\[
 \Ret(\mathcal V)=\{A^{[n]}:A\in\mathcal V\}.
\]
The EDPC assertion follows from Corollary~\ref{cor:edpc-reticulation}.
\end{proof}

\begin{remark}
Corollary~\ref{cor:edpc-gratz-spectra} is the precise mechanism by which DPC/EDPC enters the second problem of the paper.  EDPC does not provide a new Esakia duality for arbitrary residuated lattices; rather, through $n$-potency, it removes the existential exponent in reticulation and turns the distributive-lattice spectrum into the spectrum of an internal power skeleton.
\end{remark}

\begin{theorem}[Sharpness of the reticulation-image reduction]
\label{thm:reticulation-image-sharpness}
Let \(\mathcal C\) be an isomorphism-closed class of finite bounded distributive lattices.  Then there is a class \(\mathcal K_{\mathcal C}\) of finite residuated lattices such that
\[
 \Ret(\mathcal K_{\mathcal C})=\mathcal C.
\]
Equivalently, up to isomorphism, arbitrary finite bounded-distributive-lattice data can occur as a finite reticulation image.  Consequently, in the finite setting there is no non-trivial class-independent order-theoretic replacement for the invariant \(\Ret(\mathcal K)\).
\end{theorem}

\begin{proof}
Every finite bounded distributive lattice is a Heyting algebra, since relative pseudocomplements exist in finite distributive lattices.  Thus every \(D\in\mathcal C\) may be regarded as an idempotent residuated lattice by putting
\[
 x\odot y=x\wedge y.
\]
By Proposition~\ref{prop:idempotent-heyting}, the reticulation of this idempotent residuated lattice is just its bounded-lattice reduct.  Let
\[
 \mathcal K_{\mathcal C}=\{D:D\in\mathcal C\},
\]
where the members are viewed as finite idempotent residuated lattices.  Since \(\mathcal C\) is isomorphism-closed,
\[
 \Ret(\mathcal K_{\mathcal C})=\mathcal C.
\]
The final assertion follows because an arbitrary isomorphism-closed class of finite bounded distributive lattices can occur in this way.  Thus a classwise finite spectrum theorem for residuated lattices cannot be made independent of the reticulation image.
\end{proof}

\begin{remark}
Theorem~\ref{thm:reticulation-image-sharpness} is the precise sense in which the answer to the residuated-lattice version of Gr\"atzer's problem is optimal.  Theorem~\ref{thm:gratz-rl} reduces the problem to Priestley spectra of \(\Ret(\mathcal K)\); Priestley duality \cite{Priestley1970,Priestley1972} and Speed's profinite-poset viewpoint \cite{Speed1972} explain why this is the right ordered-Stone-space level.  The sharpness theorem says that, even in the finite case, \(\Ret(\mathcal K)\) can be as arbitrary as a class of finite bounded distributive lattices.  The concrete computations below should therefore be read not as samples of a hidden universal forbidden-poset theorem, but as exact calculations of the invariant in algebraically important cases.
\end{remark}

\begin{remark}[Relation with the MV-spectrum problem]
The sharpness theorem should be compared with the MV-spectrum problem.  For MV-algebras, Mundici's equivalence relates spectra to Abelian $\ell$-groups with strong unit, and the corresponding lattice-dual problem is formulated in terms of the principal $\ell$-ideal lattices $\Id_c G$.  Plo\v{s}\v{c}ica and Wehrung proved that, up to cardinality $\aleph_1$, the relevant distributive $0$-lattices are exactly the completely normal lattices with countably based differences; moreover, the cardinal bound is sharp \cite{PloscicaWehrung2024}.  This theorem is not a consequence of the finite reticulation calculations in the present paper.  Rather, it shows that in natural infinite classes the reticulation image can encode genuinely deep representation-theoretic information.
\end{remark}

\begin{corollary}[Prime-ideal version]\label{cor:ideal-version}
Let $X$ be a poset ordered as a poset of prime ideals.  Then $X$ is representable as the prime-ideal poset of the reticulation of some $A\in\mathcal K$ if and only if the order dual $X^{\partial}$ is $\mathcal K$-representable in the prime-filter sense of Definition \ref{def:representable}.
\end{corollary}

\begin{proof}
In a bounded distributive lattice, prime ideals and prime filters correspond by complementation, and inclusion of ideals is reversed under the corresponding filter order.  Thus the filter-spectrum formulation is the order dual of the ideal-spectrum formulation.
\end{proof}

\begin{corollary}\label{cor:classwise-rep}
The following classwise descriptions hold.
\begin{enumerate}
\item $\Fin\RL$-representable posets are exactly the finite posets.
\item $\Fin\MTL$- and $\Fin\BL$-representable posets are exactly the finite root systems (equivalently, finite forests in the order-theoretic convention used for Esakia spaces).
\item $\Ch\RL$-representable posets are exactly the prime spectra of bounded distributive chains.
\item $\MV_n$-representable posets are exactly the underlying ordered sets of Stone spaces with equality order, equivalently the prime spectra of Boolean algebras.  In particular, finite $\MV_n$-representable posets are exactly finite antichains.
\item If $\mathcal V$ is a Heyting variety, then $\mathcal V$-representability as residuated lattices is exactly the usual Heyting/Esakia representability for $\mathcal V$.
\end{enumerate}
\end{corollary}

\begin{proof}
Use Theorem \ref{thm:gratz-rl} together with Propositions \ref{prop:finite-image}, \ref{prop:chain-image}, \ref{prop:finite-valued-mv-ret}, Theorem \ref{thm:finite-mtl-bl-ret}, Lemma \ref{lem:finite-godel-forest}, and Corollary \ref{cor:ret-ha}.  For (i), also use Corollary \ref{cor:finite-posets}.
\end{proof}

\begin{remark}[What this solves, and what it avoids]\label{rem:what-solved}
Theorem \ref{thm:gratz-rl} gives the exact residuated-lattice reduction of Gr\"atzer's Problems 33 and 34.  Since our algebras are bounded, it directly addresses the ``has $0$ and $1$'' form of Problem 34.  It does not assert that the Heyting-algebra description by Esakia spaces and diamond systems is available in arbitrary residuated lattices.  Instead, it identifies the correct invariant: the bounded distributive lattice $\mathcal L(A)$.
\end{remark}

\section{Profinite completions and non-completion obstructions}\label{sec:completions}

We now turn to the third problem announced in the introduction: which profinite residuated lattices arise as profinite completions of abstract residuated lattices?  The completion question belongs to a broader family of problems concerning how algebraic completions interact with order-theoretic and dual-topological completions.  For Heyting algebras, MacNeille completions were described topologically by Harding and Bezhanishvili \cite{HardingBezhanishvili2004}; profinite completions and canonical extensions were compared by Bezhanishvili, Gehrke, Mines, and Morandi \cite{BezhanishviliGehrkeMinesMorandi2006}; and MacNeille, canonical, and profinite completions were systematically compared by Bezhanishvili and Vosmaer \cite{BezhanishviliVosmaer2008}.  Related positive coincidence phenomena occur for finitely generated modal algebras \cite{Vosmaer2009}, while natural extensions and profinite completions of general residually finite algebras are studied in \cite{DaveyGouveiaHaviarPriestley2011}.  In the MV-setting, Nganou's work on profinite and MacNeille completions provides the analogue closest to the non-idempotent component used below \cite{Nganou2018}.  Completion phenomena in non-distributive lattice-based structures, such as orthomodular lattices, provide a useful contrast \cite{BrunsGreechieHardingRoddy1990}.  Finally, non-rigidity of profinite completions is a major theme in group theory as well; the work of Bridson and Grunewald on Grothendieck's problem is a prominent example \cite{BridsonGrunewald2004}.  We do not claim a simple external list of all profinite residuated lattices that are completions.  Instead, we give an intrinsic criterion in terms of dense subalgebras and extension of finite-index congruences, and then combine it with a finite-tag obstruction theory: the Heyting non-completion phenomenon survives in the full residuated-lattice category and even produces genuinely non-idempotent families of examples.

Let $A$ be an abstract residuated lattice.  Its \emph{profinite completion} $\widehat A$ is the inverse limit of its finite homomorphic images, equivalently the completion of $A$ with respect to the uniformity generated by its finite-index congruences.  This finite-quotient uniformity is the residuated-lattice analogue of relative profinite uniformities studied, for example, for varieties of finite monoids \cite{PinSilva2011}.  The canonical homomorphism $\eta_A:A\to\widehat A$ need not be injective; its kernel is the intersection of all finite-index congruences of $A$.

\begin{theorem}[Intrinsic criterion for profinite completions]
\label{thm:intrinsic-completion-criterion}
Let \(P\) be a profinite residuated lattice.  The following are equivalent.
\begin{enumerate}
\item \(P\) is isomorphic, as a topological residuated lattice, to the profinite completion of some residuated lattice.
\item There is a dense residuated sublattice \(A\subseteq P\) such that every finite-index congruence \(\theta\) of \(A\) is induced by a clopen congruence \(\Theta\) of \(P\), i.e.
\[
 \theta=\Theta\cap(A\times A).
\]
\item There is a dense residuated sublattice \(A\subseteq P\) such that the uniformity induced on \(A\) from \(P\) is exactly the profinite uniformity of \(A\), generated by all finite-index congruences of \(A\).
\end{enumerate}
\end{theorem}

\begin{proof}
Assume first that \(P\cong\widehat A\).  We identify \(A\) with its canonical dense image in \(P\) after quotienting by the intersection of its finite-index congruences if necessary.  Let \(\theta\) be a finite-index congruence of \(A\).  The quotient map
\[
 q:A\to A/\theta
\]
extends uniquely to a continuous homomorphism
\[
 \widehat q:P\to A/\theta
\]
by the universal property of the profinite completion.  Its kernel \(\Theta=\ker(\widehat q)\) is a clopen congruence of \(P\), and
\[
 \theta=\Theta\cap(A\times A).
\]
Thus (i) implies (ii).

The equivalence of (ii) and (iii) is the comparison of two non-Archimedean uniformities on \(A\).  The uniformity induced from \(P\) has as a base the restrictions to \(A\) of clopen congruences of \(P\).  A clopen equivalence relation on a compact Hausdorff space has finite index, because the quotient is compact and discrete; hence these restrictions are finite-index congruences of \(A\).  Conversely, condition (ii) says that every finite-index congruence of \(A\), i.e. every basic entourage of the profinite uniformity of \(A\), is obtained in this way.  Thus the two bases generate the same uniformity.

Finally assume (iii).  Since \(P\) is compact Hausdorff and \(A\) is dense in \(P\), the completion of \(A\) with respect to the uniformity induced from \(P\) is \(P\).  By (iii), this uniformity is the profinite uniformity of \(A\), whose completion is \(\widehat A\).  Hence \(P\cong\widehat A\).
\end{proof}

\begin{remark}
Theorem~\ref{thm:intrinsic-completion-criterion} is not an external classification of all profinite residuated lattices that are completions.  It is a precise internal test: the issue is whether the finite quotients of a dense abstract subalgebra are exactly those seen by clopen congruences of the ambient profinite algebra.  This is the residuated-lattice analogue of the comparison of completion topologies that appears in the literature on MacNeille, canonical, natural, and profinite completions \cite{BezhanishviliGehrkeMinesMorandi2006,BezhanishviliVosmaer2008,Vosmaer2009,DaveyGouveiaHaviarPriestley2011,Nganou2018}.
\end{remark}

\begin{remark}[Relation with syntactic congruence methods]
The intrinsic criterion should also be compared with the syntactic- and principal-congruence methods in topological universal algebra developed by Clark--Davey--Freese--Jackson and with the axiomatizability theory of topological prevarieties of Clark--Davey--Jackson--Pitkethly \cite{ClarkDaveyFreeseJackson2005,ClarkDaveyJacksonPitkethly2008}.  In those works, profiniteness and standardness are controlled by finite determinacy of syntactic and principal congruences.  Here the obstruction is shifted to a completion level: the profinite algebra $P$ is already given, and the question is whether some dense abstract subalgebra $A\subseteq P$ has precisely the finite-index congruences induced by the clopen congruences of $P$.  Thus Theorem~\ref{thm:intrinsic-completion-criterion} can be viewed as a completion-level analogue of finite congruence control, specialized to residuated lattices, rather than as an application of syntactic-congruence theory.
\end{remark}

\begin{theorem}[Positive completion criterion in $n$-potent varieties]
\label{thm:positive-completion-npotent}
Let $\mathcal V$ be a subvariety of residuated lattices contained in an $n$-potent variety.  Let $P$ be a profinite topological residuated lattice whose underlying abstract algebra belongs to $\mathcal V$---equivalently, an inverse limit of finite discrete algebras from $\mathcal V$---and let $A\subseteq P$ be a dense subalgebra.  Suppose that every finite-index congruence of $A$ is finitely generated as a congruence.  Then
\[
 P\cong\widehat A
\]
as topological residuated lattices.
\end{theorem}

\begin{proof}
By Theorem~\ref{thm:intrinsic-completion-criterion}, it is enough to show that every finite-index congruence of $A$ is induced by a clopen congruence of $P$.  Let $\theta$ be a finite-index congruence of $A$.  By hypothesis, $\theta$ is generated by finitely many pairs $(a_i,b_i)$, $1\le i\le m$.  Under the filter--congruence isomorphism of Proposition~\ref{prop:filter-cong}, joins of congruences correspond to generated joins of filters.  Hence the congruence generated by the pairs $(a_i,b_i)$ has $1$-class generated by the corresponding equivalence elements
\[
 c_i=(a_i\to b_i)\wedge(b_i\to a_i)\qquad(1\le i\le m).
\]
Thus $F_\theta=[1]_\theta$ is the filter generated by $c_1,\ldots,c_m$.
Since $A$ is $n$-potent, this finitely generated filter is principal: putting
\[
 c=(c_1\odot\cdots\odot c_m)^n,
\]
we have
\[
 F_\theta=\uparrow_A c.
\]
Indeed, $c\le c_i$ for each $i$, so $\uparrow_A c$ contains the generators; conversely every filter containing the $c_i$ contains $c_1\odot\cdots\odot c_m$ and hence its $n$th power $c$.

The element $c$ is idempotent: since $P\in\mathcal V$ is $n$-potent,
\[
 c^2=(c_1\odot\cdots\odot c_m)^{2n}=(c_1\odot\cdots\odot c_m)^n=c.
\]
Thus the principal upset
\[
 G=\uparrow_P c
\]
is a filter of $P$.  Let $\Theta=\theta_G$.  The order relation is closed in the compact Hausdorff topological lattice $P$, so $G$ is closed.  Since the map
\[
 (x,y)\longmapsto (x\to y)\wedge(y\to x)
\]
is continuous, $\Theta$ is a closed congruence of $P$.  Moreover
\[
 G\cap A=\uparrow_P c\cap A=\uparrow_A c=F_\theta,
\]
so the filter--congruence correspondence gives
\[
 \Theta\cap(A\times A)=\theta.
\]
It remains only to check that $\Theta$ is clopen.  The quotient $P/\Theta$ is compact Hausdorff, because $\Theta$ is a closed congruence on the compact Hausdorff algebra $P$.  The image of $A$ in $P/\Theta$ is dense and is naturally isomorphic to the finite algebra $A/\theta$.  A finite subset of a Hausdorff space is closed, hence this finite dense image is the whole quotient.  Thus $P/\Theta$ is finite.  A finite Hausdorff space is discrete, so every singleton of $P/\Theta$ is clopen.  Therefore each $\Theta$-class, being the inverse image of a singleton under the quotient map $P\to P/\Theta$, is clopen in $P$.  Consequently
\[
 \Theta=\bigcup_{C\in P/\Theta} C\times C
\]
is a finite union of clopen subsets of $P^2$, hence is clopen.  Therefore $\theta$ is induced by a clopen congruence of $P$, as required.
\end{proof}

\begin{remark}[What EDPC does and does not solve]
Theorem~\ref{thm:positive-completion-npotent} explains the positive role of DPC/EDPC in the completion problem.  In an EDPC residuated-lattice variety, equivalently in an $n$-potent setting, finitely generated filters and congruences are uniformly controlled by powers.  If all finite-index congruences of a dense subalgebra are finitely generated, this finite congruence control is strong enough to make the ambient profinite algebra the full profinite completion.  But EDPC and CEP alone do not force every profinite algebra to be a profinite completion: Heyting algebras are already idempotent and have very good congruence theory, yet the non-completion examples of Bezhanishvili--Bezhanishvili--Moraschini--Stronkowski remain.  The obstruction is not failure of principal-congruence definability; it is the possibility that the finite-index congruence topology of a dense subalgebra is strictly finer than the topology induced by the ambient profinite algebra.
\end{remark}

\begin{lemma}\label{lem:completion-quotient}
Let $A$ be a residuated lattice and put
\[
 \mu_A=\bigcap\{\theta\in\Con(A):A/\theta\text{ is finite}\}.
\]
Then $\widehat A\cong\widehat{A/\mu_A}$, and $A/\mu_A$ embeds densely into $\widehat A$.
\end{lemma}

\begin{proof}
Every finite-index congruence of $A$ contains $\mu_A$, and finite-index congruences of $A/\mu_A$ correspond exactly to finite-index congruences of $A$ containing $\mu_A$.  Thus the two inverse systems of finite quotients are canonically isomorphic.  The quotient $A/\mu_A$ is residually finite, hence its canonical map into its profinite completion is injective with dense image.
\end{proof}

\begin{lemma}\label{lem:idempotent-dense}
Let $B$ be a dense subalgebra of a topological residuated lattice $H$.  If $H$ is idempotent, then $B$ is idempotent.  Hence $B$ is a Heyting algebra.
\end{lemma}

\begin{proof}
For $b\in B$, the product computed in $B$ is the restriction of the product in $H$.  Since $H$ is idempotent,
\[
 b\odot b=b.
\]
Thus $B$ is idempotent, and Proposition \ref{prop:idempotent-heyting} implies that $B$ is a Heyting algebra.
\end{proof}

\begin{theorem}[Completion transfer for profinite Heyting algebras]\label{thm:completion-transfer}
Let $H$ be a profinite Heyting algebra, regarded as a profinite residuated lattice by setting $x\odot y=x\wedge y$.  Then the following are equivalent.
\begin{enumerate}
\item $H$ is isomorphic to the profinite completion of some residuated lattice.
\item $H$ is isomorphic to the profinite completion of some Heyting algebra.
\end{enumerate}
\end{theorem}

\begin{proof}
The implication (ii) $\Rightarrow$ (i) is immediate because every Heyting algebra is a residuated lattice.

Conversely, suppose that $H\cong\widehat A$ for some residuated lattice $A$.  By Lemma \ref{lem:completion-quotient}, we may replace $A$ by the residually finite quotient $A/\mu_A$ and assume that $A$ embeds densely into $H$.  Since $H$ is idempotent, Lemma \ref{lem:idempotent-dense} implies that $A$ is idempotent, hence a Heyting algebra.  Congruences of $A$ in the Heyting signature and in the residuated-lattice signature coincide because $\odot=\wedge$.  Therefore the profinite completion of $A$ as a Heyting algebra is the same inverse limit as its profinite completion as a residuated lattice, namely $H$.
\end{proof}

The Heyting side of this question was developed in several stages.  Bezhanishvili and Bezhanishvili characterized profinite Heyting algebras in terms of image-finite posets and complete finitely approximable algebras \cite{BezhanishviliBezhanishvili2008}.  Bezhanishvili and Morandi then studied which profinite Heyting algebras arise as profinite completions \cite{BezhanishviliMorandi2009}.  Bezhanishvili, Bezhanishvili, Moraschini, and Stronkowski proved that there are profinite Heyting algebras that are not profinite completions of any Heyting algebra and characterized the Heyting varieties in which this obstruction disappears \cite{BezhanishviliEtAl2021}.  In what follows we use their result in a concrete form: their construction yields an image-finite poset $X$ for which $\Up(X)$ is a profinite Heyting algebra that is not the profinite completion of any Heyting algebra.  Theorem \ref{thm:completion-transfer} strengthens this Heyting obstruction to the category of residuated lattices.

\begin{corollary}[Non-completion profinite residuated lattices]\label{cor:noncompletion-rl}
There exist profinite residuated lattices that are not isomorphic to the profinite completion of any residuated lattice.
\end{corollary}

\begin{proof}
By the proof of \cite[Theorem~5.1]{BezhanishviliEtAl2021}---equivalently, by \cite[Theorem~6.1]{BezhanishviliEtAl2021} applied to the full variety $\HA$, which is not contained in $\DHA$---there exists an image-finite poset $X$ such that
\[
 H:=\Up(X)
\]
is a profinite Heyting algebra not isomorphic to the profinite completion of any Heyting algebra.  Regard this concrete $H$ as a residuated lattice.  If $H$ were the profinite completion of a residuated lattice, Theorem \ref{thm:completion-transfer} would imply that it is the profinite completion of a Heyting algebra, a contradiction.
\end{proof}

The previous corollary already shows that enlarging the category from Heyting algebras to residuated lattices does not remove the Heyting non-completion obstruction.  In that corollary the Heyting algebra is the specific algebra \(H=\Up(X)\) supplied by the cited Heyting construction.  We now strengthen the argument by identifying a general finite-tag obstruction mechanism.  The point is that some finite non-idempotent residuated lattices are invisible to finite Heyting quotients; multiplying by such a finite tag creates a genuinely non-idempotent profinite object without creating new finite Heyting quotients.

\subsection{Finite tags and completion reflection}

\begin{definition}[Heyting-invisible finite tag]
\label{def:heyting-invisible-tag}
A finite residuated lattice \(T\) is called a \emph{Heyting-invisible finite tag} if there is no bounded residuated-lattice homomorphism
\[
 T\to F
\]
into any non-trivial finite subdirectly irreducible Heyting algebra \(F\).
It is enough to test subdirectly irreducible finite Heyting targets for the purposes of the factorization theorem below, because every finite Heyting algebra has a finite subdirect decomposition into finite subdirectly irreducible Heyting algebras.
\end{definition}

\begin{theorem}[Finite tag factorization]
\label{thm:tag-factorization}
Let \(T\) be a Heyting-invisible finite tag, let \(H\) be a Heyting algebra, and let \(F\) be a finite Heyting algebra.  Then every bounded residuated-lattice homomorphism
\[
 f:H\times T\to F
\]
factors uniquely through the first projection
\[
 \pi_1:H\times T\to H.
\]
Equivalently, there is a unique Heyting homomorphism \(\bar f:H\to F\) such that
\[
 f=\bar f\circ\pi_1.
\]
\end{theorem}

\begin{proof}
We first prove the assertion when \(F\) is finite, non-trivial, and subdirectly irreducible.  Put
\[
 e=f(0,1_T),\qquad d=f(1_H,0_T).
\]
The elements \((0,1_T)\) and \((1_H,0_T)\) are complementary in the bounded lattice reduct of \(H\times T\).  Hence \(e\) and \(d\) are complementary elements of the Heyting algebra \(F\).  We use the standard fact that a subdirectly irreducible Heyting algebra has no non-trivial complemented elements.  Indeed, if \(0<e<1\) were complemented by \(d\), then the usual lattice decomposition map
\[
 x\longmapsto (x\wedge e,x\wedge d)
\]
would yield a non-trivial direct product decomposition of \(F\).  A direct product of two non-trivial algebras is not subdirectly irreducible, because its two coordinate kernels are non-trivial congruences with trivial intersection.  This contradicts subdirect irreducibility.  Thus \(e\in\{0,1\}\).

Assume first that \(e=0\).  Then \(d=1\).  If \(t\in T\), then \((0,t)\le (0,1_T)\), and therefore
\[
 f(0,t)\le f(0,1_T)=0.
\]
Thus \(f(0,t)=0\).  For all \(h\in H\) and \(t\in T\),
\[
 f(h,t)=f((h,0_T)\vee(0,t))=f(h,0_T)\vee f(0,t)=f(h,0_T).
\]
Moreover,
\[
 (h,0_T)=(h,1_T)\wedge(1_H,0_T),
\]
and since \(f(1_H,0_T)=1\), we obtain
\[
 f(h,0_T)=f(h,1_T)\wedge1=f(h,1_T).
\]
Therefore \(f(h,t)=f(h,1_T)\), so \(f\) factors through \(\pi_1\).  The induced map
\[
 \bar f:H\to F,\qquad \bar f(h)=f(h,1_T),
\]
is a Heyting homomorphism: the subalgebra \(H\times\{1_T\}\) is closed under \(\wedge,\vee,\odot,\to,0,1\), and in the case under consideration \(\bar f(0)=f(0,1_T)=e=0\).

It remains to rule out \(e=1\).  Then \(d=0\).  Since \((h,0_T)\le(1_H,0_T)\), we have \(f(h,0_T)=0\) for all \(h\in H\).  Hence
\[
 f(h,t)=f((h,0_T)\vee(0,t))=f(0,t).
\]
Define \(k:T\to F\) by
\[
 k(t)=f(0,t).
\]
The map \(k\) preserves \(0,1,\wedge,\vee\), and \(\odot\).  It also preserves residuals: for \(s,t\in T\),
\[
 (0,s)\to(0,t)=(1_H,s\to t),
\]
and the equality \(f(h,u)=f(0,u)\) gives
\[
 f(1_H,s\to t)=f(0,s\to t).
\]
Thus \(k\) is a bounded residuated-lattice homomorphism \(T\to F\), contradicting the assumption that \(T\) is Heyting-invisible.  Hence \(e=1\) is impossible.

If \(F\) is the one-element algebra, the assertion is immediate.  For an arbitrary non-trivial finite Heyting algebra \(F\), take a finite subdirect embedding
\[
 F\hookrightarrow \prod_{i=1}^r F_i
\]
with each \(F_i\) finite subdirectly irreducible.  For each coordinate projection \(p_i:F\to F_i\), the composite \(p_i\circ f\) factors through \(\pi_1\) by the preceding argument.  Since the maps \(p_i\) separate points, \(f\) itself is constant on the fibers of \(\pi_1\).  Hence there is a unique map \(\bar f:H\to F\) with \(f=\bar f\circ\pi_1\), and \(\bar f\) is a Heyting homomorphism because \(f\) is.  Uniqueness follows from the surjectivity of \(\pi_1\).
\end{proof}

\begin{theorem}[Completion reflection through finite tags]
\label{thm:tag-completion-reflection}
Let \(T\) be a Heyting-invisible finite tag and let \(H\) be a profinite Heyting algebra.  If \(H\times T\) is isomorphic, as a topological residuated lattice, to the profinite completion of some residuated lattice, then \(H\) is isomorphic to the profinite completion of a Heyting algebra.
\end{theorem}

\begin{proof}
Assume that \(\widehat A\cong H\times T\), and let
\[
 \eta:A\to H\times T
\]
be the canonical dense homomorphism.  Put
\[
 B=\pi_1(\eta(A))\subseteq H.
\]
The set \(B\) is a Heyting subalgebra of \(H\), since \(\pi_1\circ\eta\) is a residuated-lattice homomorphism into the idempotent residuated lattice \(H\).  It is dense because \(\eta(A)\) is dense in \(H\times T\) and \(\pi_1\) is continuous and onto.

Let \(g:B\to F\) be a homomorphism into a finite Heyting algebra \(F\).  The map
\[
 A\to F,\qquad a\mapsto g(\pi_1(\eta(a)))
\]
extends uniquely to a continuous residuated-lattice homomorphism
\[
 \widetilde g:H\times T\to F
\]
by the universal property of \(\widehat A\).  By Theorem~\ref{thm:tag-factorization},
\[
 \widetilde g=\bar g\circ\pi_1
\]
for a unique continuous Heyting homomorphism \(\bar g:H\to F\).  This extension agrees with \(g\) on \(B\): if \(b\in B\), choose \(a\in A\) with \(b=\pi_1(\eta(a))\); then
\[
 \bar g(b)=\bar g(\pi_1(\eta(a)))=\widetilde g(\eta(a))=g(\pi_1(\eta(a)))=g(b).
\]
Thus every homomorphism from \(B\) to a finite Heyting algebra extends uniquely to a continuous homomorphism from \(H\).  Equivalently, for every finite-index Heyting congruence \(\theta\) of \(B\), the quotient map \(B\to B/\theta\) extends to a continuous homomorphism \(H\to B/\theta\); hence \(\theta\) is the restriction to \(B\) of the clopen kernel of this extension.  Since \(B\) is dense in the profinite Heyting algebra \(H\), the Heyting version of the intrinsic criterion of Theorem~\ref{thm:intrinsic-completion-criterion} applies.  Therefore \(H\cong\widehat B\) as a profinite Heyting algebra.  Equivalently, the finite Heyting quotients of \(B\) are exactly the finite continuous quotients seen by \(H\).
\end{proof}

\subsection{The \L{}ukasiewicz tag and genuinely non-idempotent examples}

Let
\[
 \mathbf L_3=\{0<\alpha<1\}
\]
be the three-element \L{}ukasiewicz MV-chain.  Thus \(\alpha^2=0\) and \(\alpha\to0=\alpha\).  We regard Heyting algebras as residuated lattices by putting \(x\odot y=x\wedge y\).

\begin{lemma}[The \L{}ukasiewicz tag]
\label{lem:L3-is-tag}
The finite residuated lattice \(\mathbf L_3\) is a Heyting-invisible finite tag.
\end{lemma}

\begin{proof}
Let \(k:\mathbf L_3\to F\) be a bounded residuated-lattice homomorphism into a non-trivial Heyting algebra \(F\).  Since \(F\) is Heyting, multiplication in \(F\) is idempotent.  Hence
\[
 k(\alpha)=k(\alpha)^2=k(\alpha^2)=k(0)=0.
\]
But \(\alpha\to0=\alpha\) in \(\mathbf L_3\).  Preservation of implication gives
\[
 0=k(\alpha)=k(\alpha\to0)=k(\alpha)\to k(0)=0\to0=1,
\]
a contradiction.  Thus no such homomorphism exists.
\end{proof}

In Lemma~\ref{lem:L3-factorization}, the letter \(H\) denotes an arbitrary Heyting algebra.  The special choice \(H=\Up(X)\) is not part of the lemma; it is used only when we derive the concrete non-idempotent example in Corollary~\ref{cor:genuine-nonidempotent-noncompletion}.

Figure~\ref{fig:L3-mechanism} summarizes the mechanism.  The top layer starts from the concrete image-finite poset $X$ obtained from \cite[Theorem~5.1]{BezhanishviliEtAl2021}; its upset algebra $H=\Up(X)$ is profinite but not a Heyting profinite completion.  Passing from $H$ to $P=H\times\mathbf L_3$ adds a genuinely non-idempotent factor.  The vertical arrow from $P$ back to $H$ is the first projection $\pi_1$, and the dashed arrow indicates the content of Lemma~\ref{lem:L3-factorization}: every bounded homomorphism from $P$ to a finite Heyting algebra factors uniquely through $\pi_1$.  This corrects the misleading picture suggested by the non-homomorphic map $h\mapsto(h,1)$, whose direction is opposite to that of $\pi_1$ and which does not preserve the bottom element.

\begin{center}
\phantomsection
\refstepcounter{figure}\label{fig:L3-mechanism}
\resizebox{0.86\textwidth}{!}{%
\begin{tikzpicture}[
  >=Latex,
  font=\footnotesize,
  main/.style={draw, rounded corners=6pt, thick, inner sep=7pt, align=center},
  arr/.style={->, thick},
  darr/.style={->, thick, dashed}
]
\node[main, draw=red!70!black, fill=red!3, text width=4.35cm] (X) at (0,4.45)
  {\textbf{image-finite poset $X$}\\ not Esakia representable};
\node[main, draw=blue!70!black, fill=blue!3, text width=4.85cm] (H) at (0,2.20)
  {\textbf{$H=\Up(X)$}\\ profinite Heyting algebra\\ not an HA-completion};
\node[main, draw=green!50!black, fill=green!4, text width=4.85cm] (P) at (0,-0.55)
  {\textbf{$P=H\times\mathbf L_3$}\\ non-idempotent profinite RL\\ not an RL-completion};
\node[main, draw=orange!85!black, fill=orange!5, text width=3.55cm] (F) at (5.55,2.20)
  {\textbf{finite Heyting target $F$}\\ idempotent};
\node[main, draw=purple!75!black, fill=purple!4, text width=3.55cm] (R) at (5.55,-0.55)
  {\textbf{reflection theorem}\\ $P$ RL-completion\\ $\Rightarrow H$ HA-completion};
\node[align=center, text=orange!85!black, text width=3.6cm] at (5.55,0.78)
  {Lemma~\ref{lem:L3-factorization}:\\ $f=\bar f\circ\pi_1$};

\draw[arr, blue!70!black] (X) -- node[right] {$\Up$} (H);
\draw[arr, green!50!black] (H) -- node[right] {$\times\mathbf L_3$} (P);
\draw[arr, purple!70!black] (P.west) to[bend left=42] node[left] {$\pi_1$} (H.west);
\draw[darr, orange!85!black] (H.east) -- node[above] {$\bar f$} (F.west);
\draw[darr, orange!85!black] (P.east) -- node[below] {$f$} (F.south west);
\end{tikzpicture}
}%
\par\smallskip
\begin{minipage}{0.88\textwidth}
\small\textbf{Figure~\thefigure.} Schematic illustration of the non-idempotent non-completion mechanism.  The concrete input is an image-finite poset $X$ such that $H=\Up(X)$ is profinite but not a Heyting profinite completion.  Passing to $P=H\times\mathbf L_3$ adds a non-idempotent MV-factor.  The arrow $H\to P$ denotes formation of the product with $\mathbf L_3$, whereas $\pi_1:P\to H$ is the first projection.  Lemma~\ref{lem:L3-factorization} says that every bounded homomorphism from $P$ to a finite Heyting algebra factors through $\pi_1$, and Theorem~\ref{thm:L3-completion-reflection} uses this to reflect non-completion from $P$ back to $H$.
\end{minipage}
\end{center}

\begin{lemma}[Factorization through the Heyting coordinate]
\label{lem:L3-factorization}
Let \(H\) be a Heyting algebra, and let \(F\) be a finite Heyting algebra, regarded as an idempotent residuated lattice.  Then every bounded residuated-lattice homomorphism
\[
 f:H\times\mathbf L_3\longrightarrow F
\]
factors through the first projection
\[
 \pi_1:H\times\mathbf L_3\to H.
\]
Equivalently, there is a unique Heyting homomorphism \(\bar f:H\to F\) such that
\[
 f=\bar f\circ\pi_1.
\]
\end{lemma}

\begin{proof}
This is Theorem~\ref{thm:tag-factorization} applied to the Heyting-invisible finite tag \(\mathbf L_3\), using Lemma~\ref{lem:L3-is-tag}.
\end{proof}

\begin{theorem}[Completion reflection through \(\mathbf L_3\)]
\label{thm:L3-completion-reflection}
Let \(H\) be a profinite Heyting algebra.  If \(H\times\mathbf L_3\) is isomorphic, as a topological residuated lattice, to the profinite completion of some residuated lattice, then \(H\) is isomorphic to the profinite completion of a Heyting algebra.
\end{theorem}

\begin{proof}
This is Theorem~\ref{thm:tag-completion-reflection} applied to \(T=\mathbf L_3\).
\end{proof}

\begin{corollary}[A family of non-idempotent non-completion examples]
\label{cor:tag-noncompletion-family}
Let \(H=\Up(X)\) be a profinite Heyting algebra obtained from the construction of Bezhanishvili--Bezhanishvili--Moraschini--Stronkowski, so that \(H\) is not isomorphic to the profinite completion of any Heyting algebra.  If \(T\) is a non-idempotent Heyting-invisible finite tag, then
\[
 H\times T
\]
is a genuinely non-idempotent profinite residuated lattice which is not isomorphic to the profinite completion of any residuated lattice.
\end{corollary}

\begin{proof}
The product \(H\times T\) is profinite because \(H\) is profinite and \(T\) is finite.  It is non-idempotent because \(T\) is non-idempotent.  If \(H\times T\) were a profinite completion of a residuated lattice, Theorem~\ref{thm:tag-completion-reflection} would imply that \(H\) is the profinite completion of a Heyting algebra, contradicting the choice of \(H\).
\end{proof}

\begin{corollary}[A genuinely non-idempotent non-completion example]
\label{cor:genuine-nonidempotent-noncompletion}
There exists a profinite residuated lattice which is not Heyting and is not isomorphic to the profinite completion of any residuated lattice.  More concretely, one may take
\[
 \Up(X)\times\mathbf L_3,
\]
where \(X\) is the image-finite poset supplied by the Heyting non-completion construction.
\end{corollary}

\begin{proof}
Choose the same concrete Heyting non-completion algebra as in Corollary~\ref{cor:noncompletion-rl}: by the proof of \cite[Theorem~5.1]{BezhanishviliEtAl2021}, take an image-finite poset \(X\) such that
\[
 H:=\Up(X)
\]
is a profinite Heyting algebra not isomorphic to the profinite completion of any Heyting algebra.  Since \(\mathbf L_3\) is a non-idempotent Heyting-invisible finite tag by Lemma~\ref{lem:L3-is-tag}, Corollary~\ref{cor:tag-noncompletion-family} gives the result.  Explicitly, \(P=H\times\mathbf L_3\) is not Heyting because
\[
 (1,\alpha)^2=(1,\alpha^2)=(1,0)\ne(1,\alpha).
\]
\end{proof}

\begin{remark}
The role of \(\mathbf L_3\) is now seen as a special case of a general finite-tag mechanism.  It adds the smallest possible non-idempotent MV-component while remaining invisible to finite Heyting targets.  The proof uses factorization of finite quotients rather than a splitting of the product; in particular the tempting map \(h\mapsto(h,1)\) is irrelevant, and in the bounded language it is not a homomorphism because it does not preserve \(0\).
\end{remark}

The same transfer gives an exact variety-level statement for Heyting subvarieties viewed inside residuated lattices.

\begin{theorem}\label{thm:DHA-transfer}
Let $\mathcal V$ be a variety of Heyting algebras.  The following are equivalent.
\begin{enumerate}
\item Every profinite member of $\mathcal V$, regarded as a residuated lattice, is a profinite completion of some residuated lattice.
\item Every profinite member of $\mathcal V$ is a profinite completion of some Heyting algebra.
\item $\mathcal V\subseteq\DHA$.
\end{enumerate}
\end{theorem}

\begin{proof}
The equivalence of (ii) and (iii) is Theorem 6.1 of \cite{BezhanishviliEtAl2021}.  The equivalence of (i) and (ii) is Theorem \ref{thm:completion-transfer}, applied objectwise to profinite Heyting algebras.
\end{proof}

\begin{corollary}\label{cor:representable-varieties-transfer}
Let $\mathcal V$ be a variety of Heyting algebras.  If one asks, inside all residuated lattices, which profinite Heyting members of $\mathcal V$ are forced to be profinite completions, the answer is still governed by the Heyting condition
\[
 \mathcal V\subseteq\DHA.
\]
Moreover, the representable Heyting varieties are exactly those satisfying
\[
 \mathcal V\subseteq D_n\cap\DHA
\]
for some positive integer $n$, as in \cite[Theorem~7.6]{BezhanishviliEtAl2021}.
\end{corollary}

\begin{remark}
Theorem \ref{thm:completion-transfer} is the precise sense in which the Heyting $\DHA$ theorem survives inside residuated lattices: it survives only on the idempotent, i.e. Heyting, part of the theory.  It does not become a criterion for arbitrary residuated lattices.
\end{remark}

\section{Why the diamond-Heyting method does not generalize}\label{sec:why-dha-fails}

The results of \cite{BezhanishviliEtAl2021} give a precise answer for Heyting algebras: the relevant profinite-completion and representability phenomena are governed by diamond systems and by the variety $\DHA$.  Together with the earlier structure theory of profinite Heyting algebras \cite{BezhanishviliBezhanishvili2008,BezhanishviliMorandi2009}, this is the correct idempotent theory.  We now make explicit why it is not the general residuated-lattice answer.  Two different mechanisms are involved: finite forbidden posets stop being object-level obstructions, and reticulation forgets the nilpotent information that becomes decisive for completion questions.

\subsection{Forbidden Heyting configurations are not object-level obstructions}

The four finite posets $P_1,P_2,P_3,P_4$ used in \cite{BezhanishviliEtAl2021} axiomatize the variety $\DHA$ by the corresponding Jankov formulas.  They are obstructions in a Heyting-varietal theorem.  They are not object-level obstructions to being a profinite completion in residuated lattices.

\begin{proposition}\label{prop:forbidden-finite-not-obstruction}
Let $P$ be any finite poset.  There is a finite residuated lattice $A_P$ such that
\[
 \Spec_{\RL}(A_P)\cong P.
\]
Moreover, $A_P$ is profinite and is the profinite completion of itself.
\end{proposition}

\begin{proof}
Take $A_P=\Up(P)$, the finite bounded distributive lattice of upsets of $P$.  Since it is finite, it is a Heyting algebra and hence a residuated lattice with $\odot=\wedge$.  As in the proof of Corollary \ref{cor:finite-posets}, $P\cong\Spec(A_P)=\Spec_{\RL}(A_P)$.  Since $A_P$ is finite and discrete, it is profinite and its profinite completion is itself.
\end{proof}

\begin{corollary}\label{cor:Pi-not-RL-obstructions}
Each of the finite posets $P_1,P_2,P_3,P_4$ from the diamond-Heyting theorem occurs as the spectrum of a finite residuated lattice that is a profinite completion.  Therefore the forbidden-poset mechanism of $\DHA$ cannot be an object-level criterion for profinite completions of arbitrary residuated lattices.
\end{corollary}

\subsection{Reticulation forgets non-idempotent multiplication}

The second obstruction is more structural: reticulation records the radical distributive-lattice shadow of a residuated lattice, not the full monoidal operation.

\begin{example}\label{ex:L3}
Let $\mathbf L_3=\{0,a,1\}$ be the three-element MV-chain.  Then
\[
 a\odot a=0\ne a,
\]
so $\mathbf L_3$ is not a Heyting algebra.  Nevertheless, Lemma \ref{lem:finite-mv-chain-ret} gives
\[
 \mathcal L(\mathbf L_3)\cong\mathbf 2.
\]
Thus $\mathbf L_3$ and the two-element Boolean algebra have the same reticulation and the same one-point prime spectrum, but they are not isomorphic as residuated lattices.
\end{example}

\begin{proposition}\label{prop:ret-not-full}
The Priestley spectrum of $\mathcal L(A)$ does not determine the residuated-lattice structure of $A$ up to isomorphism, even among finite profinite residuated lattices.
\end{proposition}

\begin{proof}
The two finite residuated lattices $\mathbf L_3$ and $\mathbf 2$ have isomorphic reticulations by Example \ref{ex:L3}.  Hence their reticulation spectra are isomorphic.  But they are not isomorphic as residuated lattices, because $\mathbf 2$ is idempotent while $\mathbf L_3$ is not.
\end{proof}

\begin{remark}\label{rem:precise-difference}
This is the precise methodological difference.  In Heyting algebras, reticulation is essentially the algebra itself, and Esakia duality adds enough order-topological structure to recover implication.  In arbitrary residuated lattices, reticulation collapses nilpotent and radical information.  Therefore a general residuated-lattice Gr\"atzer problem must first pass through $\Ret(\mathcal K)$; the Heyting diamond-system analysis applies only when $\odot=\wedge$.

\end{remark}

\subsection{The product obstruction beyond the Heyting fragment}

Corollary \ref{cor:tag-noncompletion-family} gives a stronger obstruction than the idempotent transfer theorem alone.  It shows that non-completion is not confined to the Heyting subcategory and is not tied to the particular algebra \(\mathbf L_3\).  If \(T\) is any non-idempotent Heyting-invisible finite tag, the product \(H\times T\) has the same Heyting obstruction in its first coordinate, while the tag coordinate forces genuine non-idempotence.  The factorization theorem \ref{thm:tag-factorization} is the mechanism that keeps this obstruction visible to profinite completions: every finite Heyting quotient of a dense algebra mapping into \(H\times T\) must ignore the \(T\)-coordinate.

This construction also explains why reticulation alone cannot settle completion questions.  The reticulation of \(\mathbf L_3\) is \(\mathbf 2\), so the reticulation of \(H\times\mathbf L_3\) is isomorphic to the reticulation of \(H\times\mathbf 2\).  The nilpotent equation \(\alpha^2=0\), however, is crucial for the factorization argument.  Thus reticulation is the correct invariant for the Gr\"atzer spectrum problem, while profinite-completion obstructions may require additional finite algebraic tags.  This is consistent with the MV-completion literature: profinite MV-objects have a rich product structure \cite{Nganou2015}, but their reticulations can collapse to Boolean data in the finite-valued cases considered here.

\section{Conclusion}

The results of the paper can be summarized around the three profinite questions posed in the introduction.  First, Stone profiniteness for topological residuated lattices is purely topological but not merely reduct-theoretic.  The semiring reduct supplies finite clopen quotients, yet the residual structure is recovered through the cofinal lifting
\[
 \rho \longmapsto [1]_{\rho}\longmapsto \theta_{[1]_{\rho}}\subseteq \rho.
\]
Thus the inverse-limit representation is genuinely internal to residuated lattices: the finite quotients are controlled by clopen filters.

Second, the residuated-lattice form of Gr\"atzer's representation problem is governed by reticulation.  The spectrum of a residuated lattice is the Priestley spectrum of its reticulation, and classwise representability is therefore controlled by the reticulation image $\Ret(\mathcal K)$.  The sharpness theorem shows that this reduction cannot be uniformly replaced by a class-independent finite forbidden-poset description: arbitrary isomorphism-closed classes of finite bounded distributive lattices occur as finite reticulation images.  The explicit finite calculations then show that the invariant has concrete algebraic content.  Arbitrary finite posets occur for finite residuated lattices, whereas finite MTL- and BL-representable posets are precisely finite root systems, and finite-valued MV-varieties collapse at the reticulation level to Boolean data.  The $n$-potent theorem further shows that in power-bounded varieties the reticulation order is governed by a single inequality $a^n\le b$, linking the spectrum problem with the DPC/EDPC theory of substructural logics.

Third, the profinite-completion problem remains non-trivial in the full category of residuated lattices.  We do not claim a classification of all profinite residuated lattices that are completions; rather, the intrinsic completion criterion identifies the obstruction as a mismatch between the finite-index congruences of a dense abstract subalgebra and the clopen congruences induced from the profinite algebra, placing the problem in the broader universal-algebraic tradition of finite congruence control while keeping the genuinely completion-theoretic nature of the question.  The positive $n$-potent criterion identifies a substantial situation in which this mismatch disappears: in EDPC varieties, if the finite-index congruences of the dense subalgebra are finitely generated, then the ambient profinite algebra is its profinite completion.  Thus DPC/EDPC and CEP provide strong congruence control, but not a blanket completion theorem.  The Heyting non-completion obstruction does not disappear after passing to residuated lattices; in fact, the finite-tag reflection theorem turns it into a family of genuinely non-idempotent obstructions.  The key point is not reticulation, which forgets the nilpotent information in $\mathbf L_3$, but finite-target factorization: homomorphisms from $H\times T$ to finite Heyting algebras must ignore the tag coordinate whenever $T$ is Heyting-invisible.

The resulting picture is that the Heyting theorem is exactly right on the idempotent fragment, but it is not the general residuated-lattice mechanism.  Beyond Heyting algebras, reticulation is the correct invariant for spectrum representation, while profinite completion requires additional control of finite-index congruence topologies and finite non-idempotent factors.  Two natural directions remain open.  The first is to compute $\Ret(\mathcal K)$ for further algebraically significant varieties beyond the finite MTL/BL and finite-valued MV cases treated here.  The recent solution of the MV-spectrum problem in size $\aleph_1$ by Plo\v{s}\v{c}ica and Wehrung shows that, for infinite and non-locally finite classes, such computations can require substantial additional representation theory; Theorem~\ref{thm:reticulation-image-sharpness} explains why such class-sensitive information cannot be avoided in general.  The second is to classify finite Heyting-invisible tags and determine how their internal algebraic features affect the resulting families of genuinely non-idempotent non-completion examples.  A related completion-theoretic direction is to identify varieties or dense subalgebras in which finite-index congruences are automatically finitely generated, since Theorem~\ref{thm:positive-completion-npotent} then turns EDPC into a positive profinite-completion theorem.

\end{document}